\documentclass[12pt,reqno]{amsart}

\usepackage{graphicx}
\usepackage[english]{babel}
\usepackage{a4wide}
\usepackage{pgf}
\usepackage{subfigure}

\numberwithin{equation}{section}

\newtheorem{theorem}{Theorem}[section]
\newtheorem{lemma}[theorem]{Lemma}

\theoremstyle{definition}
\newtheorem{definition}[theorem]{Definition}

\theoremstyle{remark}
\newtheorem{remark}[theorem]{Remark}

\numberwithin{equation}{section}

\newcommand{\Nedelec}{N{\'{e}}d{\'{e}}lec}

\newcommand{\Hormander}{H{\"{o}}rmander}

\newcommand{\im}{\mathbf{i}}
\newcommand{\normal}{\boldsymbol{n}}

\begin{document}

\title[FOSLS for the Helmholtz equation]{A first order system least squares method for the Helmholtz equation}

\author{Huangxin Chen}
\address{School of Mathematical Sciences and Fujian Provincial Key Laboratory on Mathematical Modeling and High Performance Scientific Computing, Xiamen University, Fujian, 361005, China} 
\email{chx@xmu.edu.cn} 

\author{Weifeng Qiu}
\address{Department of Mathematics, City University of Hong Kong, 83 Tat Chee Avenue, Kowloon, Hong Kong, China} 
\email{weifeqiu@cityu.edu.hk}
\thanks{Corresponding author: Weifeng Qiu (weifeqiu@cityu.edu.hk)}

\begin{abstract}
We present a first order system least squares (FOSLS) method for the Helmholtz equation 
at high wave number $k$, which always leads to a Hermitian positive definite algebraic system. 
By utilizing a non-trivial solution decomposition to the dual FOSLS problem which is quite different from that of 
the standard finite element methods, we give an error analysis to the $hp$-version of the FOSLS method where the 
dependence on the mesh size $h$, the approximation order $p$, and the wave number $k$ is given explicitly. 
In particular, under some assumption of the boundary of the domain, the $L^2$ norm error estimate of the scalar 
solution from the FOSLS method is shown to be quasi optimal under  the condition that $kh/p$ is sufficiently small 
and the polynomial degree $p$ is at least $O(\log k)$. Numerical experiments are given to verify the theoretical results.
\end{abstract}

\thanks{We are grateful to Professor Markus Melenk for valuable discussions on the theory in this paper. The first author 
would like to thank the supports from the National Natural Science Foundation of China (Grant No. 11201394) and
the Natural Science Foundation of Fujian Province (Grant No. 2013J05016). The work of the second author was 
partially supported by a grant from the Research Grants Council of the Hong Kong Special Administrative Region, 
China (Project No. CityU 109713).}



\maketitle

\section{Introduction}
Lots of least squares methods have been extensively studied for the efficient and accurate numerical approximation of many 
partial differential equations such as the elliptic, elasticity and Stokes equations. As mentioned in \cite{cai06}, there are 
three kinds of least-squares methods: the inverse approach, the div approach, and the div-curl approach. The interest of 
this paper is to consider the div approach least squares method which applies a chosen $L^2$ norm to a natural first order 
system for the Helmholtz equation with Robin boundary condition which is the first order approximation of the radiation condition:  
\begin{subequations}
\label{hwn_pde}
\begin{align}
\label{hwn_pde1}
 -\Delta u - k^{2} u & = f\quad\text{ in }\Omega,\\
\label{hwn_pde2}
 \dfrac{\partial u}{\partial \normal} - \im k u & = g \quad\text{ on }\partial \Omega,  
\end{align}
\end{subequations}
where $\Omega\subset \mathbb{R}^{d}$ ($d=2\text{ or } 3$) is a bounded, Lipschitz and connected domain, 
the wave number $k$ is real and positive, and $\im$ denotes the imaginary unit. 
We want to point out that if the sign before ${\bf i}$ in (\ref{hwn_pde2}) is positive, the corresponding 
least squares method and theoretical analysis in this paper also hold. 
We impose further assumptions on the domain $\Omega$ in the following:
\begin{enumerate}
\item[(A1)] There is a constant $C>0$ such that for any $f\in L^{2}(\Omega)$ and $g\in L^{2}(\partial\Omega)$, 
the Helmholtz equation (\ref{hwn_pde}) has a unique solution $u\in H^{1}(\Omega)$ satisfying 
\begin{align*}
\Vert \nabla u\Vert_{L^{2}(\Omega)} + k\Vert u\Vert_{L^{2}(\Omega)} \leq 
C \left(\Vert f\Vert_{L^{2}(\Omega)} + \Vert g\Vert_{L^{2}(\partial\Omega)} \right).
\end{align*}
\item[(A2)] The boundary of $\Omega$ is analytic.
\end{enumerate}
The above assumptions are intrinsic for the analysis in this paper, while the least squares method 
can be applied for more general cases. In fact, \cite{Hetmaniuk07} shows the assumption (A1) 
holds if the domain $\Omega$ is star-shaped; and \cite[Theorem~$1.8$]{Baskin2015} obtains 
the same estimate without the star-shaped restriction.

Due to the well-known pollution effect for the numerical solution of the Helmholtz equation, the standard 
Galerkin finite element methods can maintain a desired accuracy only if the mesh resolution is appropriately 
increased. Numerous nonstandard methods have been proposed in the literature to obtain more stable and 
accurate approximation, which includes quasi-stabilized finite element methods \cite{Babu}, absolutely 
stable discontinuous Galerkin (DG) methods \cite{Wu,Wuhp,feng,GM2011,clx2013}, continuous interior 
penalty finite element methods \cite{Wu2011,Wu2013}, the partition of unity finite element methods \cite{Melenk1,Melenk2}, 
the ultra weak variational formulation \cite{CD}, plane wave DG methods \cite{Amara,Hiptmair}, spectral methods \cite{Shen}, 
generalized Galerkin/finite element methods \cite{Babu-0,Melenk}, meshless methods \cite{Babu-1}, 
and the geometrical optics approach \cite{Engquist}. 

Generally, the linear systems from most of the above nonstandard Galerkin finite element approximations of the 
Helmholtz equation with high wave number $k$ are strongly indefinite. But the least-squares Galerkin method for 
the Helmholtz equation always yields a Hermitian positive definite system \cite{Chang90,Lee00}. Hence it attracts 
the design of an efficient solver. For instance, a div-curl approach least squares method was applied to the Helmholtz 
equation in \cite{Lee00}, and an efficient solver based on wave-ray multigrid was proposed. 
Recently,  numerical results in \cite{GopalSchob:2014:preconditioner} show that a multiplicative 
Schwarz algorithm, without coarse solver, provides a $p$-preconditioner for solving the DPG system. 
The numerical observations suggest that the condition number of the preconditioned system is independent of the wavenumber $k$ and the polynomial degree $p$. Since both DPG methods and FOSLS are residual minimization methods such that their linear systems are Hermitian positive definite, it is promising that the multiplicative Schwarz preconditioner in \cite{GopalSchob:2014:preconditioner} will provide similar preconditioning for FOSLS. We will show the effect of the multiplicative Schwarz preconditioner for our FOSLS in a separate paper.

A key result revealed by J.M.~Melenk and S.~Sauter in \cite{MelenkSIAM11} is that the polynomial degree $p$ 
should be chosen in a wavenumber-dependent way to yield optimal convergent conditions. This important 
result was analyzed based on the standard Galerkin finite element method. It shows that, under the assumption 
that the solution operator for Helmholtz problems is polynomially bounded in $k$, quasi optimal convergence 
can be obtained under the conditions that $kh/p$ is sufficiently small and the polynomial degree $p$ is at least 
${\rm O}({\rm log}\, k)$.

An objective of this paper is to extend the key result in \cite{MelenkSIAM11} to the div approach FOSLS method, which will be 
called FOSLS method for brevity in the following. We use the standard Raviart-Thomas finite element space and continuous 
piece-wise polynomial finite element space for the discretization of the FOSLS method. The stability of the FOSLS solutions 
for the Helmholtz equation can be obtained by the property of FOSLS formulation and a Rellich-type identity approach. 
The main difficulty in the analysis lies in the establishment of quasi optimal convergence for the FOSLS method. 
We first mimic the technique proposed in \cite{MelenkSIAM11} to decompose the Helmholtz solution into an oscillatory 
analytic part and a nonoscillatory elliptic part. A key estimate for the oscillatory analytic part of the Helmholtz solution 
(cf. (\ref{decomp_ineq3}) in Theorem~\ref{thm_divergence_regularity}) is further derived for the error analysis of the FOSLS method 
for the Helmholtz equation. Another crucial estimate lies in the derivation of the dependence of convergence on the 
polynomial degree $p$. A new $H(\text{div})$ projection is designed to overcome this problem, and some important estimates, 
which reveal the dependence of the projection error on $k,h,p$, for this $H(\text{div})$ projection are obtained. 
In Remark~\ref{remark_rt}, we explain why it is necessary to use Raviart-Thomas space instead of vector valued 
continuous piece-wise polynomial space to approximate vector fields in $H(\text{div},\Omega)$.
In Remark~\ref{remark_div_space}, we give detailed explanation why the projection-based interpolation 
in \cite{Demko:2008:PBI} can {\em not} be applied for the quasi optimal convergent estimate for the Helmholtz equation. 
The most important part of the analysis lies in a modified duality argument for the FOSLS method which is motivated by 
the duality argument used in \cite{cai06}. Roughly speaking, the corresponding dual FOSLS problem is to find 
$(\boldsymbol{\psi}, v)\in \{ \boldsymbol{\psi}\in  H(\text{div},\Omega): \boldsymbol{\psi}
\cdot\normal |_{\partial \Omega} \in L^{2}(\partial\Omega) \}\times H^{1}(\Omega)$ satisfying 
\begin{align*}
& \Vert u - u_{h}\Vert_{L^{2}(\Omega)}^{2} \\
= & (\im k (\boldsymbol{\phi}-\boldsymbol{\phi}_{h})+\nabla (u - u_{h}), \im k \boldsymbol{\psi} + \nabla v)_{\Omega}\\
& + (\im k (u - u_{h})+ \nabla\cdot (\boldsymbol{\phi} - \boldsymbol{\phi}_{h}), \im k v 
+ \nabla\cdot \boldsymbol{\psi})_{\Omega}\\
& + k\langle (\boldsymbol{\phi}-\boldsymbol{\phi}_{h})\cdot \normal + (u - u_{h}), 
\boldsymbol{\psi}\cdot \normal + v \rangle_{\partial\Omega}.
\end{align*}
Here, $\im k \boldsymbol{\phi} + \nabla u = 0$, and $(\boldsymbol{\phi}_{h}, u_{h})$ is the numerical 
approximation to $(\boldsymbol{\phi}, u)$. Then, the regularity estimates for the oscillatory analytic part 
$(\boldsymbol{\psi}_{A}, v_{A})$ and the nonoscillatory elliptic part $(\boldsymbol{\psi}_{H^{2}}, v_{H^{2}})$ 
of the solution of the above dual FOSLS problem are deduced. Since the above dual FOSLS problem is quite 
different from the dual problem used in \cite{MelenkJSC13, MelenkSIAM11}, these regularity estimates, 
especially the estimate of $\Vert \nabla\cdot\boldsymbol{\psi}_{H^{2}}\Vert_{H^{1}(\Omega)}$ 
(cf. (\ref{test_func_ineq5}) in Lemma~\ref{lemma_test_function}), gets involved with 
non-trivial modification to the original proof of solution decomposition in \cite{MelenkSIAM11}.
Finally the quasi optimality of the $L^2$ norm error estimate for the scalar solution of the FOSLS method 
for the Helmholtz equation can be finally obtained under the conditions that $kh/p$ is sufficiently small 
and the polynomial degree $p$ is at least ${O}({\rm log}\, k)$. 

We want to emphasize that FOSLS is closely related to the discontinuous Petrov-Galerkin (DPG) methods,
see \cite{Bouma2014,Broersen2014,cai2015,Vcalo_Plates,Carstensen2014,ChanHeuerTanDemkowicz:DPG_CD,ChanEvansQiu, 
CF,DemkoGopal:2010:DPG1,DemkoGopal:2010:DPG2,DemkoGopal:2013:DPG3,
Ellis_localDPG, GopaQiu:PracticalDPG,Heuer_transit,Heuer_trace,Roberts_Stokes}.
Recently, the DPG$_\varepsilon$ method, which is of the least-squares type, was proposed in \cite{Jay2013}. 
The DPG$_\varepsilon$ solution may yield less pollution error than the general FOSLS with fixed polynomial degree $p$ 
and on the same mesh. The analysis for FOSLS in this paper can be useful to develop and analyze pollution free 
DPG methods. In addition, the implementation of DPG methods have been significantly simplified in \cite{Roberts_DPGcode}.

The organization of the paper is as follows: We introduce some notation, the FOSLS method, and the main result in the next section. 
Section 3 is devoted to the proof of the stability estimate of the FOSLS method for the Helmholtz equation.
In Section 4, we present some auxiliary results for the regularity estimates of the oscillatory analytic part and the nonoscillatory 
elliptic part of the Helmholtz solution, and the approximation properties of the finite element spaces. The regularity estimates of 
the solution to the dual  FOSLS problem and the proof of a quasi optimal convergent result of this paper are stated in Section 5. 
In the final section, we give some numerical results to confirm our theoretical analysis.

\section{The first order system least squares method, and main results}

\subsection{Geometry of the mesh}
We describe the meshes we are going to use. We first introduce the concept of generalized cell.
Next, we define a $C^{0}$-compatible mesh and then the so-called quasi-uniform regular meshes 
which are the meshes we are going to work with. Finally, we propose a way to generate
quasi-uniform regular meshes.

\subsubsection{Reference cells and curved cells}
We denote by $\widehat{K}$ the reference cell in $\mathbb{R}^{d}$.
This closed set is the standard unit tetrahedron when $d=3$. It is 
the standard unit triangle when $d=2$. We denote by $\bigtriangleup_{m}(\widehat{K})$ 
the collection of all $m$-dimensional subcells of $\widehat{K}$ for $0\leq m \leq d-1$. 
They are all faces of $\widehat{K}$ when $m=2$, are all edges of $\widehat{K}$ when $m=1$, 
and all vertexes of $\widehat{K}$ when $m=0$.

\begin{definition}
\label{generalzied_cell}
A closed subset $K$ of $\mathbb{R}^{d}$ is a generalized $d$-dimensional cell 
if there is a $C^{1}$-diffeomorphism $G_{K}$ 
from the reference cell $\widehat{K}$ to $K$ such that $G_{K}\in C^{\infty}(\widehat{K})$. 
\end{definition}

We denote by $h_{K}$ the diameter of $K$. We also denote by $\bigtriangleup_{m}(K)$ the collection
of all $m$-dimensional subcells of $K$, which are exactly $G_{K}(\bigtriangleup_{m}(\widehat{K}))$. 
Note that all points $x$ in $K$ are of the form $x=G_{K}(\widehat{x})$ 
where $\widehat{x}$ lies in $\widehat{K}$.

\subsubsection{$C^{0}$-compatible mesh}
We denote by $\mathcal{T}_{h}$ the finite collection of generalized cells in 
$\mathbb{R}^{d}$ such that for any two different 
generalized cells $K,K^{\prime}\in \mathcal{T}_{h}$, either $K\cap K^{\prime}=\emptyset$ or $K\cap K^{\prime}\in 
\bigtriangleup_{m}(K)\cap\bigtriangleup_{m}(K^{\prime})$ for some $0\leq m\leq d-1$. Here, the parameter $h$ is the 
maximum of the diameters $h_K$ of the cells $K$ in $\mathcal{T}_{h}$. 

We denote by  $\bigtriangleup_{d-1}(\mathcal{T}_{h})$ the collection of $\bigtriangleup_{d-1}(K)$ for all cells 
$K$ in $\mathcal{T}_{h}$. Notice that for any $F\in \bigtriangleup_{d-1}(\mathcal{T}_{h})$, either $F=K\cap K^{\prime}$ 
with $K,K^{\prime}\in \mathcal{T}_{h}$ or $F\subset\partial\Omega$ where $\Omega$ is an open subset in $\mathbb{R}^{d}$ 
such that $\overline{\Omega}=\cup_{K\in\mathcal{T}_{h}}K$. 

\begin{definition}
\label{C0_comp}
We say that $\mathcal{T}_{h}$ is a $C^{0}$-compatible mesh if,
for any two subcells $\widehat{F},\widehat{F}^{\prime}\in \bigtriangleup_{d-1}(\widehat{K})$,
where $K,K^\prime\in \mathcal{T}_h$,  
such that $G_{K}(\widehat{F})=G_{K^{\prime}}(\widehat{F}^{\prime})$, 
there is an affine mapping $\mathcal{R} :\widehat{F}\rightarrow\widehat{F}^{\prime}$ 
satisfying
\begin{equation}
\label{C0_comp_eq}
G_{K}|_{\widehat{F}}=G_{K^{\prime}}|_{\widehat{F}^{\prime}}\circ\mathcal{R}. \end{equation} 

We call $K$ an element of $\mathcal{T}_{h}$. And, we call $F\in\bigtriangleup_{d-1}(\mathcal{T}_{h})$ 
a face in $\mathcal{T}_{h}$. 
\end{definition}
The $C^{0}$-compatible mesh is introduced in \cite{hpbook}.

\subsubsection{The quasi-uniform regular meshes}
We use the symbol $\nabla^{n}$ to denote derivatives of order $n$; more precisely, for a function 
$u:\Omega \rightarrow \mathbb{R}, \Omega\subset\mathbb{R}^{d}$, we define 
\begin{align}
\label{high_der}
\vert \nabla^{n} u (x)\vert^{2} = \Sigma_{\alpha\in \mathbb{N}_{0}^{d}:\vert \alpha\vert = n} \dfrac{n!}{\alpha !}
\vert D^{\alpha} u (x)\vert^{2}. 
\end{align}
Here, $\mathbb{N}_{0}$ is the set of all non-negative integers.
Now, we are ready to give the description of meshes we are going to use in this paper.
\begin{definition}(quasi-uniform regular meshes)
\label{k_regular_meshes}
Let $\{\mathcal{T}_{h}\}_{h\in\mathrm{I}}$ be a family of $C^{0}$-compatible meshes. We call 
$\{\mathcal{T}_{h}\}_{h\in\mathrm{I}}$ a family of quasi-uniform regular meshes if for any 
$h\in\mathrm{I}$ and any $K\in\mathcal{T}_{h}$, 
\[
\sup_{\widehat{x}\in\widehat{K}}\Vert (\nabla G_{K}(\widehat{x}))^{-1}\Vert
\leq C_{G}h^{-1},\qquad
\sup_{\widehat{x}\in\widehat{K}}\Vert \nabla^{i}G_{K}(\widehat{x})\Vert
\leq C_{G}h^{i}\gamma^{i} i!\quad \forall i\geq 0,
\]
where $C_{G}, \gamma$ are a positive constants independent of $h$ and of $K$, 
$\Vert \cdot\Vert$ is the Euclidean norm.
\end{definition}

Throughout this paper, we assume that the domain $\Omega$ admits a family of 
quasi-uniform regular meshes $\{\mathcal{T}_{h}\}_{h\in\mathrm{I}}$
such that $\overline{\Omega}=\cup_{K\in\mathcal{T}_{h}}K$ for any $h\in\mathrm{I}$.
As usual, we can always pick an $h$ in $\mathrm{I}$ arbitrarily
close to zero.

\subsubsection{Isoparametric refinement}\label{proceeding_refinement}
Next, we present a way of generating a family of quasi-uniform regular meshes for $\Omega$.
 We begin by obtaining a $C^{0}$-compatible mesh for $\Omega$, $\mathcal{T}_{h_0}$,  and by setting
$G^0 := \{ G^0_K,\; \forall K \in \mathcal{T}_{h_0} \}$. 
To obtain a finer mesh $\mathcal{T}_{h_1}$, we first divide the reference element 
$\widehat{K}$ uniformly into elements $\widehat{K}'$. Then we refine the actual element $K$ 
via the mapping $G^0_K$, that is $K' = G^0_K(\widehat{K}')$. 
The remaining meshes are obtained by repeating this process. It is not difficult to verify that 
the family of meshes obtained in this manner is quasi-uniform regular if we have that
\[
\sup_{\widehat{x}\in\widehat{K}}\Vert (\nabla G_{K}^{0}(\widehat{x}))^{-1}\Vert
\leq C_{G}h_{K}^{-1},\qquad
\sup_{\widehat{x}\in\widehat{K}}\Vert \nabla^{i}G_{K}^{0}(\widehat{x})\Vert
\leq C_{G}h_{K}^{i}\gamma^{i} i!\quad \forall i\geq 0,
\]
for any $K\in\mathcal{T}_{h_0}$. We emphasize that the meshes satisfying \cite[Assumption~$5.2$]{MelenkMC10} 
are quasi-uniform regular.

\subsection{First order system least squares method}
We define complex valued vector field space and scalar function space
\begin{equation}
\label{sol_space}
\boldsymbol{V} = \{\boldsymbol{\phi}\in H(\text{div}, \Omega): \boldsymbol{\phi}|_{\partial\Omega} 
\in L^{2}(\partial \Omega; \mathbb{R}^{d}) \},\qquad 
W = H^{1}(\Omega).
\end{equation}

For any mesh $\mathcal{T}_{h}$ and any $p\geq 0$, we denote by 
\begin{subequations}
\label{sol_space_discrete}
\begin{align}
\label{sol_space_discrete1}
\boldsymbol{V}_{h} & = \{\boldsymbol{\phi}\in H(\text{div},\Omega):  
\det(DG_{K})DG_{K}^{-1}\left( \boldsymbol{\phi}|_{K}\circ G_{K} \right) \in \boldsymbol{RT}_{p+1}(\widehat{K}) 
\text{ for any }K\in\mathcal{T}_{h} \},\\
\label{sol_space_discrete2}
W_{h} & = \{ v \in W: v|_{K}\circ G_{K}\in P_{p+1}(\widehat{K}) \text{ for any }K\in\mathcal{T}_{h}\},
\end{align}
\end{subequations}
where $\boldsymbol{RT}_{p+1}(\widehat{K})=P_{p+1}(\widehat{K};\mathbb{R}^{d}) + \boldsymbol{x} P_{p+1}(\widehat{K})$ 
and $P_{p+1}(\widehat{K})$ are complex valued Raviart-Thomas space and 
complex valued polynomial with order up to $p+1$, respectively. Notice that $\boldsymbol{V}_{h}\subset\boldsymbol{V}$ 
and the restriction of $\boldsymbol{V}_{h}$ on each element $K$ is exactly $\boldsymbol{RT}_{p+1}(\widehat{K})$ mapped 
onto $K$ via the Piola transform corresponding to $G_{K}$.

The least squares functional is defined as 
\begin{align*}
& R((\boldsymbol{\phi}, u); (f ,g)) \\
= & \Vert \im k \boldsymbol{\phi} + \nabla u \Vert_{L^{2}(\Omega)}^{2} 
+ \Vert \im k u + \nabla\cdot \boldsymbol{\phi} + \im f k^{-1}  \Vert_{L^{2}(\Omega)}^{2}\\
&\quad + \Vert k^{1/2} (\boldsymbol{\phi}\cdot\normal + u - k^{-1/2} g) \Vert_{L^{2}(\partial\Omega)}^{2}
\quad \forall (\boldsymbol{\phi}, u)\in \boldsymbol{V}\times W.
\end{align*}
The first order system least squares (FOSLS) method is to find 
$(\boldsymbol{\phi}_{h}, u_{h})\in \boldsymbol{V}_{h}\times W_{h}$ by 
\begin{align}
\label{fosls_hwn}
&  b((\boldsymbol{\phi}_{h}, u_{h}), (\boldsymbol{\psi}, v))\\
\nonumber
=& (-\im fk^{-1}, \im k v+ \nabla\cdot \boldsymbol{\psi})_{\Omega} 
+ \langle \im g,  \boldsymbol{\psi}\cdot\normal + v\rangle_{\partial \Omega}\quad\forall (\boldsymbol{\psi},v)\in 
\boldsymbol{V}_{h}\times W_{h}.
\end{align}
Here, for any $(\boldsymbol{\phi}, u), (\boldsymbol{\psi}, v)\in \boldsymbol{V}\times W$, 
\begin{align}
\label{biform_hwn}
& b((\boldsymbol{\phi}, u), (\boldsymbol{\psi}, v)) \\
\nonumber
= & (\im k \boldsymbol{\phi}+\nabla u, \im k \boldsymbol{\psi} + \nabla v)_{\Omega} 
+ (\im k u + \nabla\cdot \boldsymbol{\phi}, \im k v + \nabla\cdot \boldsymbol{\psi})_{\Omega}
 + k \langle \boldsymbol{\phi}\cdot\normal + u, \boldsymbol{\psi}\cdot\normal + v \rangle_{\partial\Omega}.
\end{align}
For any complex valued functions $u$ and $v$, we define 
\begin{align*}
(u, v)_{\Omega} = \int_{\Omega} u \bar{v} \qquad \langle u, v \rangle_{\partial \Omega} 
= \int_{\partial\Omega} u \bar{v}.
\end{align*}
If $u\in H^{1}(\Omega)$ is the solution of the Helmholtz equation (\ref{hwn_pde}), then 
$(\boldsymbol{\phi}= \im k^{-1}\nabla u , u)$ satisfies 
\begin{equation}
\label{fosls_hwn_pde}
b((\boldsymbol{\phi}, u), (\boldsymbol{\psi}, v)) = 
 (-\im fk^{-1}, \im k v+ \nabla\cdot \boldsymbol{\psi})_{\Omega} 
+ \langle \im g,  \boldsymbol{\psi}\cdot\normal + v\rangle_{\partial \Omega}
\quad\forall (\boldsymbol{\psi},v)\in \boldsymbol{V}\times W.
\end{equation}

We notice that the FOSLS method (\ref{fosls_hwn}) is very similar to the one in \cite{Lee00} except that 
we use the Raviart-Thomas space to approximate $\boldsymbol{\phi}$ and the Robin boundary condition 
is imposed weakly. In Remark~\ref{remark_rt}, we explain why it is necessary to use Raviart-Thomas space 
instead of vector valued continuous piece-wise polynomial space to approximate vector fields in $H(\text{div},\Omega)$. 
In Remark~\ref{remark_weight}, we explain why we weight with the factor $k$ to the inner product 
$\langle \boldsymbol{\phi}\cdot\normal + u, \boldsymbol{\psi}\cdot\normal + v \rangle_{\partial\Omega}$ 
on the boundary $\partial\Omega$.

\subsection{Main result}
We outline the main result in the following by showing the stability and 
the quasi optimality of the FOSLS method for the Helmholtz equation.

\begin{theorem} (Stability)
\label{thm_stability}
We assume that the assumption (A1) holds.
There is a constant $C$, which is independent of the wave number $k\geq k_{0}>0$, such that  
\begin{align*}
\Vert \boldsymbol{\phi}\Vert_{L^{2}(\Omega)}^{2} + \Vert u\Vert_{L^{2}(\Omega)}^{2}
+ k\Vert \boldsymbol{\phi}\cdot\normal + u \Vert_{L^{2}(\partial \Omega)}^{2}
\leq C b((\boldsymbol{\phi}, u), (\boldsymbol{\phi}, u))\quad 
\forall (\boldsymbol{\phi}, u)\in \boldsymbol{V}\times W.
\end{align*}
\end{theorem}

\begin{theorem} (Quasi optimal convergence)
\label{main_res}
\label{thm_duality}
We assume that the Assumptions (A1, A2) hold. $(\boldsymbol{\phi}_{h}, u_{h})$ is the solution of 
the FOSLS method (\ref{fosls_hwn_pde}).
There are constants $c_{1}, c_{2}, C>0$ independent of $h,p$ and $k\geq k_{0}>0$ such that if 
\begin{align}
\label{p_condition}
\dfrac{kh}{p} < c_{1} \text{ together with } p\geq c_{2} (\log k + 1),
\end{align}
then for any $(\boldsymbol{\psi}_{h}, v_{h})\in \boldsymbol{V}_{h}\times W_{h}$, we have 
\begin{align*}
& \Vert u-u_{h} \Vert_{L^{2}(\Omega)} \\
\leq & C h \left( k\Vert \boldsymbol{\phi} - \boldsymbol{\psi}_{h}\Vert_{L^{2}(\Omega)} 
+ \Vert\nabla\cdot (\boldsymbol{\phi} - \boldsymbol{\psi}_{h}) \Vert_{L^{2}(\Omega)} 
+ \Vert \nabla (u - v_{h})\Vert_{L^{2}(\Omega)} 
+ k \Vert u - v_{h}\Vert_{L^{2}(\Omega)} \right)\\
& \quad + C h^{1/2} \Vert (\boldsymbol{\phi} - \boldsymbol{\psi}_{h})\cdot\normal\Vert_{L^{2}(\partial\Omega)}.
\end{align*}
\end{theorem}

\section{Stability}
We give the proof of the stability estimate (Theorem~\ref{thm_stability}) for 
the solution of the FOSLS method in the above section.

\begin{proof} (Theorem~\ref{thm_stability})
We define
\begin{align*}
\boldsymbol{\eta} & = \im k \boldsymbol{\phi}+\nabla u \quad 
w = \im k u + \nabla\cdot \boldsymbol{\phi}\text{ in }\Omega,\qquad 
\mu = \boldsymbol{\phi}\cdot\normal + u \text{ on }\partial\Omega.
\end{align*}
We consider two problems 
\begin{align}
\label{stab_prob1}
\im k \boldsymbol{\phi}_{1} + \nabla u_{1} & = 0 \quad \text{ in } \Omega,\\
\nonumber
\im k u_{1} + \nabla\cdot\boldsymbol{\phi}_{1} & = w \quad \text{ in } \Omega,\\
\nonumber
\boldsymbol{\phi}_{1}\cdot\normal + u_{1} & = \mu \quad \text{ on } \partial\Omega,
\end{align}
and 
\begin{align}
\label{stab_prob2}
\im k \boldsymbol{\phi}_{2} + \nabla u_{2} & = \boldsymbol{\eta} \quad \text{ in } \Omega,\\
\nonumber
\im k u_{2} + \nabla\cdot\boldsymbol{\phi}_{2} & = 0 \quad \text{ in } \Omega,\\
\nonumber
\boldsymbol{\phi}_{2}\cdot\normal + u_{2} & = 0 \quad \text{ on } \partial\Omega.
\end{align}

According to the assumption (A1), there is a unique solution $u_{1}\in H^{1}(\Omega)$ of the following problem 
\begin{align*}
& -\Delta u_{1} - k^{2} u_{1} = \im k w \quad \text{ in } \Omega,\\
& \dfrac{\partial u_{1}}{\partial \normal} - \im k u_{1} = -\im k \mu \quad \text{ on } \partial\Omega,
\end{align*}
and 
\begin{align*}
\Vert \nabla u_{1}\Vert_{L^{2}(\Omega)} + k\Vert u_{1}\Vert_{L^{2}(\Omega)} \leq 
C k \left(\Vert w\Vert_{L^{2}(\Omega)} + \Vert \mu\Vert_{L^{2}(\partial\Omega)} \right).
\end{align*}
We define $\boldsymbol{\phi}_{1}$ by $\im k \boldsymbol{\phi}_{1} + \nabla u_{1} = 0$ in $\Omega$. 
Then, we have that $(\boldsymbol{\phi}_{1}, u_{1})$ is a solution of the problem (\ref{stab_prob1}) such that 
$\boldsymbol{\phi}_{1} \in \{\psi\in H(\text{div}, \Omega): \psi\cdot\normal|_{\partial\Omega} \in L^{2}(\partial\Omega)\}$  
and 
\begin{align}
\label{stab_ineq1}
\Vert \boldsymbol{\phi}_{1}\Vert_{L^{2}(\Omega)} + k^{-1}\Vert \nabla u_{1}\Vert_{L^{2}(\Omega)} 
+\Vert u_{1}\Vert_{L^{2}(\Omega)} \leq C\left(\Vert w\Vert_{L^{2}(\Omega)} 
+ \Vert \mu\Vert_{L^{2}(\partial\Omega)}\right).
\end{align}

According to \cite[Lemma~$4.3$]{DemkoGopalMugaZitelli2012} and the assumption (A1) again, 
there is a solution $(\boldsymbol{\phi}_{2}, u_{2})\in H(\text{div}, \Omega)\times H^{1}(\Omega)$  
of the problem (\ref{stab_prob2}) satisfying 
\begin{align}
\label{stab_ineq2}
\Vert \boldsymbol{\phi}_{2}\Vert_{L^{2}(\Omega)} + k^{-1}\Vert \nabla u_{2}\Vert_{L^{2}(\Omega)} 
+\Vert u_{2}\Vert_{L^{2}(\Omega)} \leq C\Vert \boldsymbol{\eta}\Vert_{L^{2}(\Omega)}.
\end{align}

It is easy to see that 
\begin{align*}
\im k (\boldsymbol{\phi}_{1}+\boldsymbol{\phi}_{2}) + \nabla (u_{1}+u_{2}) & = \boldsymbol{\eta} \quad \text{ in } \Omega,\\
\im k (u_{1}+u_{2}) + \nabla\cdot(\boldsymbol{\phi}_{1}+\boldsymbol{\phi}_{2}) & = w \quad \text{ in } \Omega,\\
(\boldsymbol{\phi}_{1}+\boldsymbol{\phi}_{2})\cdot\normal + (u_{1}+u_{2}) & = \mu \quad \text{ on } \partial\Omega.
\end{align*}
By the assumption (A1), $u_{1}+u_{2} = u$ and $\boldsymbol{\phi}_{1}+\boldsymbol{\phi}_{2}=\boldsymbol{\phi}$. 
Then, by (\ref{stab_ineq1}) and (\ref{stab_ineq2}), we can conclude that the proof is complete.
\end{proof}

\begin{remark}
\label{remark_stability}
By the same argument in the above proof, for any $(\boldsymbol{\phi}, u)\in \boldsymbol{V}\times W$,
\begin{align*}
&  C_{0}\left(\Vert \boldsymbol{\phi}\Vert_{L^{2}(\Omega)}^{2} + \Vert u\Vert_{L^{2}(\Omega)}^{2} 
+ \Vert \boldsymbol{\phi}\cdot\normal + u\Vert_{L^{2}(\partial\Omega)}^{2}\right)\\
\leq & (\im k \boldsymbol{\phi}+\nabla u, \im k \boldsymbol{\phi}+\nabla u)_{\Omega}
+ (\im k u + \nabla\cdot \boldsymbol{\phi}, \im k u + \nabla\cdot \boldsymbol{\phi})_{\Omega}
+ \langle \boldsymbol{\phi}\cdot\normal + u, \boldsymbol{\phi}\cdot\normal + u \rangle_{\partial\Omega}. 
\end{align*}
\end{remark}

\section{Auxiliary results}
In this section, we provide some auxiliary results. 

\subsection{Decomposition of the Helmholtz  solution}
\label{section_decomp}
The main results of this section are Theorem~\ref{thm_divergence_regularity} and Lemma~\ref{lemma_analytic}. 
Theorem~\ref{thm_divergence_regularity} is the same as \cite[Theorem~$4.10$]{MelenkSIAM11} except 
(\ref{decomp_ineq3}). We emphasize that (\ref{decomp_ineq3}) is essential in the proof of duality argument 
(cf. (\ref{test_func_ineq5}) in Lemma~\ref{lemma_test_function}) of the first order system least squares method 
for the Helmholtz equation. We require the Assumption (A2) holds throughout section~\ref{section_decomp}.

We need to recall several notations in \cite{MelenkSIAM11}. We denote by $\mathcal{F}$ the Fourier transform 
for functions in $L^{2}(\mathbb{R}^{d})$. For functions $f\in L^{2}(\mathbb{R}^{d})$, the high frequency filter 
$H_{\mathbb{R}^{d}}$ and the low frequency filter $L_{\mathbb{R}^{d}}$ are defined by
\begin{align}
\label{global_filters}
\mathcal{F} (L_{\mathbb{R}^{d}} f) = \chi_{\eta k} \mathcal{F}(f), \quad
\mathcal{F} (H_{\mathbb{R}^{d}} f) = (1-\chi_{\eta k}) \mathcal{F}(f),
\end{align}
where $\chi_{\eta k}$ is the characteristic function of the ball $B_{\eta k}(0)$ and $\eta$ is a positive parameter 
which will be determined later. Let $E_{\Omega}:L^{2}(\Omega)\rightarrow L^{2}(\mathbb{R}^{d})$ be 
the Stein extension operator \cite[Chapter VI]{Stein70}. Then, for $f\in L^{2}(\Omega)$, we define 
\begin{align}
\label{domain_filters}
L_{\Omega} f = (L_{\mathbb{R}^{d}} (E_{\Omega} f))|_{\Omega}, \quad 
H_{\Omega} f = (H_{\mathbb{R}^{d}} (E_{\Omega} f))|_{\Omega}.
\end{align}
We denote by $G^{N}$ a lifting operator with the mapping property 
$G^{N}: H^{s}(\partial\Omega)\rightarrow H^{3/2+s}(\Omega)$ for any $s>0$ and $\partial_{\normal}G^{N} g = g$. 
As mentioned in \cite[Remark~$4.1$]{MelenkSIAM11}, we can choose $G^{N}$ independent of $k$. 
We then define $H_{\partial\Omega}^{N}$ and $L_{\partial\Omega}^{N}$ by
\begin{align}
\label{boundary_filters}
H_{\partial\Omega}^{N}(g) = \partial_{\normal}H_{\Omega}(G^{N}(g)),\quad 
L_{\partial\Omega}^{N}(g) = \partial_{\normal}L_{\Omega}(G^{N}(g)).
\end{align}
We denote by $N_{k}$ the Newton potential operator defined in \cite[$(4.11)$]{MelenkSIAM11}. 
We define $S_{k}: (f, g)\rightarrow u$ to be the solution operator of the Helmholtz equation (\ref{hwn_pde}),
and $S_{k}^{\Delta}: g \rightarrow u$ to be the solution operator of the modified Helmholtz equation with 
Robin boundary conditions; i.e., $u=S_{k}^{\Delta}(g)$ solves
\begin{align}
\label{k_poisson}
-\Delta u + k^{2} u = 0\text{ in }\Omega, \quad \partial_{\normal}u-\im k u = g \text{ on }\partial\Omega.
\end{align}

\begin{lemma}
\label{lemma_decomp_source_term}
We assume that the Assumptions (A1, A2) hold. Let $q\in (0,1)$. Then, there are constants 
$C,\gamma>0$ independent of $k$ such that for any $f\in L^{2}(\Omega)$, the function $u=S_{k}(f,0)$ 
can be written as $u = u_{A} + u_{H^{2}}+\tilde{u}$, where
\begin{align*}
& k\Vert u_{A} \Vert_{L^{2}(\Omega)} + \Vert \nabla u_{A} \Vert_{L^{2}(\Omega)} \leq C \Vert f\Vert_{L^{2}(\Omega)},\\
& \Vert \nabla^{p+2} u_{A}\Vert_{L^{2}(\Omega)} \leq Ck^{-1}\gamma^{p} \max (p+2, k)^{p+2}\Vert f\Vert_{L^{2}(\Omega)}
\quad \forall p\geq 0,\\ 
& \Vert \Delta u_{A} + k^{2} u_{A}\Vert_{H^{1}(\Omega)} \leq C k \Vert f\Vert_{L^{2}(\Omega)},\\
& k\Vert u_{H^{2}}\Vert_{L^{2}(\Omega)} + \Vert \nabla u_{H^{2}}\Vert_{L^{2}(\Omega)} \leq q k^{-1} \Vert f\Vert_{L^{2}(\Omega)},\\
& \Vert u_{H^{2}}\Vert_{H^{2}(\Omega)} \leq C \Vert f\Vert_{L^{2}(\Omega)},
\end{align*}
and the remainder $\tilde{u} = S_{k}(\tilde{f},0)$ satisfies
\begin{align*}
-\Delta \tilde{u} - k^{2}\tilde{u} = \tilde{f},\quad \partial_{\normal}\tilde{u} - \im k \tilde{u}|_{\partial\Omega} = 0, 
\end{align*}
where
\begin{align*}
\Vert \tilde{f}\Vert_{L^{2}(\Omega)} \leq q\Vert f\Vert_{L^{2}(\Omega)}.
\end{align*}
\end{lemma}

\begin{proof}
According to Theorem~\ref{thm_stability} and \cite[Lemma~$4.15$]{MelenkSIAM11}, it is easy to see that 
except $\Vert \Delta u_{A} + k^{2} u_{A}\Vert_{H^{1}(\Omega)} \leq C k \Vert f\Vert_{L^{2}(\Omega)}$, all other statements hold.

In order to prove  $\Vert \Delta u_{A} + k^{2} u_{A}\Vert_{H^{1}(\Omega)} \leq C k \Vert f\Vert_{L^{2}(\Omega)}$, 
we need to go through the proof of \cite[Lemma~$4.15$]{MelenkSIAM11}. It is shown that 
\begin{align*}
u_{A} = u_{A}^{I} + u_{A}^{II},\quad u_{A}^{I} = S_{k}(L_{\Omega}f, 0),\quad 
u_{H^{2}}^{I} = N_{k}(H_{\Omega}f),\quad u_{A}^{II} = S_{k}(0, L_{\partial\Omega}^{N}
(\im k u_{H^{2}}^{I}-\partial_{\normal}u_{H^{2}}^{I})). 
\end{align*}
So, it is sufficient to show that
\begin{align*}
\Vert \Delta u_{A}^{I} + k^{2} u_{A}^{I}\Vert_{H^{1}(\Omega)} \leq C k \Vert f\Vert_{L^{2}(\Omega)},\quad 
\Vert \Delta u_{A}^{II} + k^{2} u_{A}^{II}\Vert_{H^{1}(\Omega)} \leq C k \Vert f\Vert_{L^{2}(\Omega)}.
\end{align*}

Since $u_{A}^{I} = S_{k}(L_{\Omega}f, 0)$, we have 
\begin{align*}
& \Vert \Delta u_{A}^{I} + k^{2} u_{A}^{I}\Vert_{H^{1}(\Omega)}^{2} = \Vert L_{\Omega}f\Vert_{H^{1}(\Omega)}^{2}
\leq C \int_{\mathbb{R}^{d}} (1+\vert \xi\vert^{2} )\vert \widehat{L_{\Omega}f}\vert^{2}\\
= & C \int_{\mathbb{R}^{d}} (1+\vert \xi\vert^{2} )\chi_{\eta k}\vert \widehat{E_{\Omega}f}\vert^{2}
\leq C k^{2} \Vert E_{\Omega} f\Vert_{L^{2}(\Omega)}^{2} \leq C k^{2} \Vert f\Vert_{L^{2}(\Omega)}^{2}.
\end{align*}
On the other hand, $ u_{A}^{II} = S_{k}(0, L_{\partial\Omega}^{N}
(\im k u_{H^{2}}^{I}-\partial_{\normal}u_{H^{2}}^{I}))$ implies that 
$\Delta u_{A}^{II} + k^{2} u_{A}^{II}=0$. So, we can conclude that 
$\Vert \Delta u_{A} + k^{2} u_{A}\Vert_{H^{1}(\Omega)} \leq C k \Vert f\Vert_{L^{2}(\Omega)}$.
\end{proof}

\begin{lemma}
\label{lemma_decomp_boundary_data}
We assume that the Assumptions (A1, A2) hold. Let $q\in (0,1)$. Then, there are constants 
$C,\gamma>0$ independent of $k$ such that for any $g\in H^{1/2}(\partial\Omega)$, the function $u=S_{k}(0,g)$ 
can be written as $u = u_{A} + u_{H^{2}} + \tilde{u}$, where
\begin{align*}
& k\Vert u_{A} \Vert_{L^{2}(\Omega)} + \Vert \nabla u_{A} \Vert_{L^{2}(\Omega)} \leq C \Vert g\Vert_{H^{1/2}(\partial\Omega)},\\
& \Vert \nabla^{p+2} u_{A}\Vert_{L^{2}(\Omega)} \leq Ck^{-1}\gamma^{p} \max (p+2, k)^{p+2}
\Vert g\Vert_{H^{1/2}(\partial\Omega)}\quad \forall p\geq 0,\\ 
& \Vert \Delta u_{A} + k^{2} u_{A}\Vert_{H^{1}(\Omega)} \leq C k \Vert g\Vert_{H^{1/2}(\partial\Omega)},\\
& k\Vert u_{H^{2}}\Vert_{L^{2}(\Omega)} + \Vert \nabla u_{H^{2}}\Vert_{L^{2}(\Omega)} \leq q k^{-1} 
\Vert g\Vert_{H^{1/2}(\partial\Omega)},\\
& \Vert u_{H^{2}}\Vert_{H^{2}(\Omega)} \leq C \Vert g\Vert_{H^{1/2}(\partial\Omega)},
\end{align*}
and the remainder $\tilde{u} = S_{k}(0,\tilde{g})$ satisfies
\begin{align*}
-\Delta \tilde{u} - k^{2}\tilde{u} = 0,\quad \partial_{\normal}\tilde{u} - \im k \tilde{u}|_{\partial\Omega} = \tilde{g}, 
\end{align*}
where
\begin{align*}
\Vert \tilde{g}\Vert_{H^{1/2}(\partial\Omega)} \leq q\Vert g\Vert_{H^{1/2}(\partial\Omega)}.
\end{align*}
\end{lemma}

\begin{proof}
According to Theorem~\ref{thm_stability} and \cite[Lemma~$4.16$]{MelenkSIAM11}, it is easy to see that except 
$\Vert \Delta u_{A} + k^{2} u_{A}\Vert_{H^{1}(\Omega)} \leq C k \Vert g\Vert_{H^{1/2}(\partial\Omega)}$, all other statements hold.

In order to prove  $\Vert \Delta u_{A} + k^{2} u_{A}\Vert_{H^{1}(\Omega)} \leq C k \Vert g\Vert_{H^{1/2}(\partial\Omega)}$, 
we need to go through the proof of \cite[Lemma~$4.16$]{MelenkSIAM11}. It is shown that 
\begin{align*}
u_{A} = u_{A}^{I} + u_{A}^{II},\quad u_{A}^{I} = S_{k}(0, L_{\partial\Omega}^{N}g),\quad 
u_{H^{2}}^{I} = S_{k}^{\Delta}(H_{\partial\Omega}^{N}g),\quad u_{A}^{II} = S_{k}(L_{\Omega}(2 k^{2} u_{H^{2}}^{I}),0). 
\end{align*}
So, it is sufficient to show that
\begin{align*}
\Vert \Delta u_{A}^{I} + k^{2} u_{A}^{I}\Vert_{H^{1}(\Omega)} \leq C k \Vert g\Vert_{H^{1/2}(\partial\Omega)},\quad 
\Vert \Delta u_{A}^{II} + k^{2} u_{A}^{II}\Vert_{H^{1}(\Omega)} \leq C k \Vert g\Vert_{H^{1/2}(\partial\Omega)}.
\end{align*}

Since $u_{A}^{II} = S_{k}(L_{\Omega}(2 k^{2} u_{H^{2}}^{I}),0)$, we have 
\begin{align*}
& \Vert \Delta u_{A}^{II} + k^{2} u_{A}^{II}\Vert_{H^{1}(\Omega)}^{2} = 
\Vert L_{\Omega}(2 k^{2} u_{H^{2}}^{I})\Vert_{H^{1}(\Omega)}^{2}
= 2k^{2}\int_{\mathbb{R}^{d}} (1 + \vert \xi\vert^{2})\vert \widehat{L_{\Omega}( u_{H^{2}}^{I})}\vert^{2}\\
= & 2k^{2}\int_{\mathbb{R}^{d}} (1 + \vert \xi\vert^{2})\chi_{\eta k} \vert \widehat{ E_{\Omega} u_{H^{2}}^{I}}\vert^{2}
\leq C k^{4} \Vert E_{\Omega} u_{H^{2}}^{I}\Vert_{L^{2}(\Omega)}^{2}\leq C k^{4} \Vert u_{H^{2}}^{I}\Vert_{L^{2}(\Omega)}^{2}\\
\leq & C k^{2}\Vert g\Vert_{H^{1/2}(\partial\Omega)}^{2}.
\end{align*}
The last inequality above is obtained by \cite[$(4.31)$]{MelenkSIAM11}. On the other hand, $u_{A}^{I} = S_{k}(0, L_{\partial\Omega}^{N}g)$ 
implies that $\Delta u_{A}^{I} + k^{2} u_{A}^{I}=0$. So, we can conclude that 
$\Vert \Delta u_{A} + k^{2} u_{A}\Vert_{H^{1}(\Omega)} \leq C k \Vert g\Vert_{H^{1/2}(\partial\Omega)}$.
\end{proof}

\begin{theorem}
\label{thm_divergence_regularity}
We assume that the Assumptions (A1, A2) hold. Then, there are constants 
$C,\gamma>0$ independent of $k\geq k_{0}$ such that for any 
$(f,g)\in L^{2}(\Omega)\times H^{1/2}(\partial\Omega)$, the function $u=S_{k}(f,g)$ 
can be written as $u = u_{A} + u_{H^{2}}$, where
\begin{subequations}
\label{decomp_ineqs}
\begin{align}
\label{decomp_ineq1}
& k\Vert u_{A} \Vert_{L^{2}(\Omega)} + \Vert \nabla u_{A} \Vert_{L^{2}(\Omega)} 
\leq C \left(\Vert f\Vert_{L^{2}(\Omega)}+\Vert g\Vert_{H^{1/2}(\partial\Omega)}\right),\\
\label{decomp_ineq2}
& \Vert \nabla^{p+2} u_{A}\Vert_{L^{2}(\Omega)} \leq Ck^{-1}\gamma^{p} \max (p+2, k)^{p+2}
\left(\Vert f\Vert_{L^{2}(\Omega)}+ \Vert g\Vert_{H^{1/2}(\partial\Omega)}\right)\quad \forall p\geq 0,\\
\label{decomp_ineq3}
& \Vert \Delta u_{A} + k^{2} u_{A}\Vert_{H^{1}(\Omega)} \leq C k 
\left(\Vert f\Vert_{L^{2}(\Omega)}+ \Vert g\Vert_{H^{1/2}(\partial\Omega)}\right),\\
\label{decomp_ineq4}
& \Vert u_{H^{2}}\Vert_{H^{2}(\Omega)} + k \Vert u_{H^{2}}\Vert_{H^{1}(\Omega)} 
+ k^{2}\Vert u\Vert_{L^{2}(\Omega)} \leq C \left(\Vert f\Vert_{L^{2}(\Omega)}+\Vert g\Vert_{H^{1/2}(\partial\Omega)}\right).
\end{align}
\end{subequations}
\end{theorem}

\begin{proof}
By proceeding in the same way as the proof of \cite[Theorem~$4.10$]{MelenkSIAM11} 
with Lemma~\ref{lemma_decomp_source_term} and Lemma~\ref{lemma_decomp_boundary_data}, 
we can conclude that the proof is complete.
\end{proof}

\begin{lemma}
\label{lemma_analytic}
We assume that the Assumptions (A1, A2) hold.
Then, there are constants $C,\gamma>0$ independent of $k\geq k_{0}$ such that for 
any analytic functions $\tilde{f}$ and $\tilde{g}$ in $\Omega$, 
\begin{align*}
& \Vert \nabla^{p+2} S_{k}(\tilde{f}, \tilde{g})\Vert_{L^{2}(\Omega)} \leq 
C \gamma^{p+2} k^{-1}\max (p+2, k)^{p+2} \left(\Vert \tilde{f}\Vert_{L^{2}(\Omega)} 
+ \Vert \tilde{g}\Vert_{H^{1}(\Omega)}  \right),\\
& k\Vert  S_{k}(\tilde{f}, \tilde{g})\Vert_{L^{2}(\Omega)} + \Vert \nabla S_{k}(\tilde{f}, \tilde{g}) \Vert_{L^{2}(\Omega)} 
\leq C \left( \Vert \tilde{f}\Vert_{L^{2}(\Omega)} + \Vert \tilde{g}\Vert_{H^{1}(\Omega)} \right).
\end{align*}
\end{lemma}

\begin{proof}
We denote by $v = S_{k}(\tilde{f}, \tilde{g})$.
By the Assumption (A1), we have immediately
\begin{align*}
k\Vert  v\Vert_{L^{2}(\Omega)} + \Vert \nabla v \Vert_{L^{2}(\Omega)} 
\leq C \left( \Vert \tilde{f}\Vert_{L^{2}(\Omega)} + \Vert \tilde{g}\Vert_{H^{1}(\Omega)} \right).
\end{align*}
By the Assumption (A2) and \cite[Theorem~$4.18$(ii)]{McleanEllipticBook}, 
it is easy to see $v\in C^{\infty}(\Omega)$.

In order to show the other estimate, we follow the main steps in the proof of \cite[Lemma~$4.13$]{MelenkSIAM11}.
We take $\epsilon = k^{-1}$. It is easy to see that $v$ satisfies
\begin{align*}
-\epsilon^{2}\Delta v - v = \epsilon^{2}\tilde{f} \text{ in }\Omega,\qquad 
\epsilon^{2} \partial_{\normal} v = \epsilon \left(\epsilon \tilde{g} + \im v \right) \text{ on } \partial \Omega.
\end{align*}
Then, by applying \cite[Proposition~$5.4.5$ and Remark~$5.4.6$]{MelenkBook02} to the above equation, 
we can conclude that the proof is complete.
\end{proof}

  \subsection{Approximation properties of finite element spaces}
We would like to show approximation properties of some projection operators for finite element spaces 
(\ref{sol_space_discrete}).

We define a projection $\widehat{\Pi}_{\boldsymbol{V}} : H^{1}(\widehat{K};\mathbb{R}^{d} )\rightarrow 
\boldsymbol{RT}_{p+1}(\widehat{K})$ by 
\begin{subequations}
\label{proj_div}
\begin{align}
\label{proj_div_eq1}
& \langle (\widehat{\Pi}_{\boldsymbol{V}}\widehat{\boldsymbol{\psi}} - \widehat{\boldsymbol{\psi}})
\cdot\widehat{\normal}, \widehat{\mu} \rangle_{\widehat{F}} = 0, \quad 
\forall \widehat{\mu} \in P_{p+1}(\widehat{F}), \widehat{F}\in \bigtriangleup_{d-1}(\widehat{K}),\\
\label{proj_div_eq2}
& ( \widehat{\Pi}_{\boldsymbol{V}}\widehat{\boldsymbol{\psi}} - \widehat{\boldsymbol{\psi}}, 
\widehat{\nabla}\times \widehat{\boldsymbol{\phi}}  )_{\widehat{K}} = 0,\quad  
\forall \widehat{\boldsymbol{\phi}} \in \boldsymbol{Q}_{p+1,0}(\widehat{K}),\\
\label{proj_div_eq3}
& \Vert \widehat{\nabla}\cdot ( \widehat{\Pi}_{\boldsymbol{V}}\widehat{\boldsymbol{\psi}} 
- \widehat{\boldsymbol{\psi}}) \Vert_{L^{2}(\widehat{K})}\rightarrow \min.
\end{align}
\end{subequations}
Here, $\boldsymbol{Q}_{p+1,0}(\widehat{K}) = \{\widehat{\boldsymbol{\phi}} \in P_{p+2}^{-}
\Lambda^{d-2}(\widehat{K}): \text{tr} \widehat{\boldsymbol{\phi}} |_{\partial \widehat{K}} = 0 \}$.
When $d=3$, $P_{p+2}^{-}\Lambda^{d-2}(\widehat{K})$ is the \Nedelec  1st-kind 
$H(\text{curl})$ element of degree $\leq p+1$. When $d=2$, $P_{p+2}^{-}\Lambda^{d-2}(\widehat{K})$ 
is the Lagrange element of degree $\leq p+2$. 

We emphasize that the projection (\ref{proj_div}) is the same as the projection \cite[$(201)$]{Demko:2008:PBI} 
except the way to impose normal component on the boundary of $\widehat{K}$ (see the difference between 
(\ref{proj_div_eq1}) and the first condition in \cite[$(201)$]{Demko:2008:PBI}).

\begin{remark}
\label{remark_div_space}
Since we use the Raviart-Thomas space for the approximation to functions in $\boldsymbol{V}$, 
the natural idea is to utilize the projection-based interpolation $\Pi^{\text{div}}$ in \cite[$(201)$]{Demko:2008:PBI}. 
We notice that in \cite[Theorem~$5.3$]{Demko:2008:PBI}, the estimate of approximation error 
$\Vert \Pi^{\text{div}}\boldsymbol{\psi} - \boldsymbol{\psi}\Vert_{H(\text{div},\Omega)}$ gets involved with 
$\Vert \boldsymbol{\psi}\Vert_{H^{r}(\text{div},\Omega)}$ where $r>0$ and the Sobolev norm 
$\Vert \cdot \Vert_{H^{r}(\text{div},\Omega)}$ is defined in \Hormander's style. 
When $r$ is a non-negative integer, the norm $\Vert \cdot \Vert_{H^{r}(\text{div},\Omega)}$ in 
\Hormander's style is equivalent to $\left(\Sigma_{0\leq i\leq r} \Vert \nabla^{i} \cdot 
\Vert_{L^{2}(\Omega)}^{2}\right)^{1/2}$ which is provided in Lemma~\ref{lemma_test_function}. 
However, it is {\em not} obvious to see how the equivalent constants depend on $r$. 
So, we introduce projection (\ref{proj_div}) and give the following Lemma~\ref{lemma_proj_div_reference}.
\end{remark}

\begin{lemma}
\label{lemma_proj_div_reference}
There is a constant $C>0$ such that for any $\widehat{\boldsymbol{\psi}} \in H^{1}(\widehat{K};\mathbb{R}^{d})$,
\begin{align*}
& \Vert  \widehat{\Pi}_{\boldsymbol{V}}\widehat{\boldsymbol{\psi}} 
- \widehat{\boldsymbol{\psi}}\Vert_{H(\text{div},\widehat{K})} \\
\leq & C \left( \inf_{\widehat{\boldsymbol{\phi}}\in \boldsymbol{RT}_{p+1}(\widehat{K})} 
\Vert \widehat{\boldsymbol{\psi}} - \widehat{\boldsymbol{\phi}} \Vert_{H(\text{div}, \hat{K})} 
+ \inf_{\widehat{\boldsymbol{\varphi}}\in \boldsymbol{RT}_{p+1}(\widehat{K})} 
\Vert (\widehat{\boldsymbol{\varphi}} - \widehat{\boldsymbol{\psi}})\cdot\widehat{\normal} 
\Vert_{L^{2}(\partial\widehat{K})}\right).
\end{align*} 
In addition, we have 
\begin{align}
\label{commutativity_div_reference}
\widehat{\nabla}\cdot \widehat{\Pi}_{\boldsymbol{V}}\widehat{\boldsymbol{\psi}} 
= \widehat{P}\widehat{\nabla}\cdot\widehat{\boldsymbol{\psi}},\qquad 
(\widehat{\Pi}_{\boldsymbol{V}}\widehat{\boldsymbol{\psi}})\cdot\widehat{\normal} |_{\widehat{F}} 
= \widehat{P}_{\widehat{F}} (\widehat{\boldsymbol{\psi}}\cdot\widehat{\normal} |_{\widehat{F}})\quad 
\forall \widehat{F}\in \bigtriangleup_{d-1}(\widehat{K}).
\end{align}
Here, $\widehat{P}$ and $\widehat{P}_{\widehat{F}}$ are the standard $L^{2}$-orthogonal projections 
onto $P_{p+1}(\widehat{K})$ and $P_{p+1}(\widehat{F})$, respectively.
\end{lemma}

\begin{proof}
(\ref{commutativity_div_reference}) can be verified straightforwardly 
by the definition of $\widehat{\Pi}_{\boldsymbol{V}}$. In the following, 
we give the proof of the inequality for the case $d=3$, 
which is similar to that of \cite[Theorem~$5.3$]{Demko:2008:PBI}.

We define $\widehat{P}^{\text{div}}: H^{1}(\widehat{K};\mathbb{R}^{3})\rightarrow 
\boldsymbol{RT}_{p+1}(\widehat{K})$ by
\begin{align*}
& ( \widehat{P}^{\text{div}}\widehat{\boldsymbol{\psi}} - \widehat{\boldsymbol{\psi}}, 
\widehat{\nabla}\times \widehat{\boldsymbol{\phi}}  )_{\widehat{K}} = 0,\quad  
\forall \widehat{\boldsymbol{\phi}} \in P_{p+2}^{-}\Lambda^{1}(\widehat{K}),\\
& \Vert \widehat{\nabla}\cdot ( \widehat{P}^{\text{div}}\widehat{\boldsymbol{\psi}} 
- \widehat{\boldsymbol{\psi}}) \Vert_{L^{2}(\widehat{K})}\rightarrow \min.
\end{align*}
$\widehat{P}^{\text{div}}$ is introduced in \cite[$(198)$]{Demko:2008:PBI}.

We denote by $\boldsymbol{q} = \widehat{\Pi}_{\boldsymbol{V}}\widehat{\boldsymbol{\psi}} 
- \widehat{P}^{\text{div}}\widehat{\boldsymbol{\psi}}$. Then, 
\begin{align*}
& \boldsymbol{q}\cdot\widehat{\normal} = (\widehat{\Pi}_{\boldsymbol{V}}\widehat{\boldsymbol{\psi}} 
- \widehat{P}^{\text{div}}\widehat{\boldsymbol{\psi}})\cdot\widehat{\normal}\text{ on }\partial\widehat{K},\\
& ( \boldsymbol{q}, \widehat{\nabla}\times \widehat{\boldsymbol{\phi}}  )_{\widehat{K}} = 0,\quad  
\forall \widehat{\boldsymbol{\phi}} \in \boldsymbol{Q}_{p+1,0}(\widehat{K}),\\
& \Vert \widehat{\nabla}\cdot \boldsymbol{q}\Vert_{L^{2}(\widehat{K})}\rightarrow \min.
\end{align*}
We define $\boldsymbol{p}\in \boldsymbol{RT}_{p+1}(\widehat{K})$ by
\begin{align*}
& \boldsymbol{p}\cdot\widehat{\normal} = (\widehat{\Pi}_{\boldsymbol{V}}\widehat{\boldsymbol{\psi}} 
- \widehat{P}^{\text{div}}\widehat{\boldsymbol{\psi}})\cdot\widehat{\normal}\text{ on }\partial\widehat{K},\\
& ( \boldsymbol{p}, \widehat{\nabla}\times \widehat{\boldsymbol{\phi}}  )_{\widehat{K}} = 0,\quad  
\forall \widehat{\boldsymbol{\phi}} \in \boldsymbol{Q}_{p+1,0}(\widehat{K}),\\
& \Vert \boldsymbol{p}\Vert_{H(\text{div},\widehat{K})}\rightarrow \min.
\end{align*}
We claim that $\boldsymbol{p}$ satisfies 
\begin{subequations}
\label{p_props}
\begin{align}
\label{p_prop1}
& \boldsymbol{p}\cdot\widehat{\normal} = (\widehat{\Pi}_{\boldsymbol{V}}\widehat{\boldsymbol{\psi}} 
- \widehat{P}^{\text{div}}\widehat{\boldsymbol{\psi}})\cdot\widehat{\normal}\text{ on }\partial\widehat{K},\\
\label{p_prop2}
& \Vert \boldsymbol{p}\Vert_{H(\text{div},\widehat{K})}\rightarrow \min,\\
\label{p_prop3}
&\Vert \boldsymbol{q}\Vert_{H(\text{div},\widehat{K})} \leq 
C \Vert \boldsymbol{p}\Vert_{H(\text{div},\widehat{K})}.
\end{align}
\end{subequations}

Let $\boldsymbol{\varepsilon}_{\widehat{K}}^{\text{div}}$ be the polynomial extension operator 
in \cite[Theorem~$7.1$]{DemkoGopalSchob:2009:EOP3}. Then, by (\ref{p_props}), we have 
\begin{align*}
& \Vert \widehat{\Pi}_{\boldsymbol{V}}\widehat{\boldsymbol{\psi}} 
- \widehat{P}^{\text{div}}\widehat{\boldsymbol{\psi}}\Vert_{H(\text{div},\widehat{K})}
=\Vert \boldsymbol{q}\Vert_{H(\text{div},\widehat{K})} \leq C \Vert \boldsymbol{p}\Vert_{H(\text{div},\widehat{K})}\\
\leq & C \Vert \boldsymbol{\varepsilon}_{\widehat{K}}^{\text{div}} ( (\widehat{\Pi}_{\boldsymbol{V}}\widehat{\boldsymbol{\psi}} 
- \widehat{P}^{\text{div}}\widehat{\boldsymbol{\psi}})\cdot\widehat{\normal}|_{\partial\widehat{K}} )\Vert_{H(\text{div},\widehat{K})}\\
\leq & C \Vert (\widehat{\Pi}_{\boldsymbol{V}}\widehat{\boldsymbol{\psi}} 
- \widehat{P}^{\text{div}}\widehat{\boldsymbol{\psi}})\cdot\widehat{\normal}\Vert_{H^{-1/2}(\partial \widehat{K})}\\
\leq & C (\Vert (\widehat{P}^{\text{div}}_{\boldsymbol{V}}\widehat{\boldsymbol{\psi}} 
- \widehat{\boldsymbol{\psi}})\cdot\widehat{\normal}\Vert_{H^{-1/2}(\partial \widehat{K})}
+\Vert (\widehat{\Pi}_{\boldsymbol{V}}\widehat{\boldsymbol{\psi}} 
- \widehat{\boldsymbol{\psi}})\cdot\widehat{\normal}\Vert_{H^{-1/2}(\partial \widehat{K})})\\
\leq & C (\Vert \widehat{P}^{\text{div}}_{\boldsymbol{V}}\widehat{\boldsymbol{\psi}} 
- \widehat{\boldsymbol{\psi}}\Vert_{H(\text{div}, \widehat{K})}
+\Vert (\widehat{\Pi}_{\boldsymbol{V}}\widehat{\boldsymbol{\psi}} 
- \widehat{\boldsymbol{\psi}})\cdot\widehat{\normal}\Vert_{H^{-1/2}(\partial \widehat{K})})\\
\leq & C (\Vert \widehat{P}^{\text{div}}_{\boldsymbol{V}}\widehat{\boldsymbol{\psi}} 
- \widehat{\boldsymbol{\psi}}\Vert_{H(\text{div}, \widehat{K})}
+\Vert (\widehat{\Pi}_{\boldsymbol{V}}\widehat{\boldsymbol{\psi}} 
- \widehat{\boldsymbol{\psi}})\cdot\widehat{\normal}\Vert_{L^{2}(\partial \widehat{K})}).
\end{align*}
Then, by combining the above inequality with \cite[Theorem~$5.2$]{Demko:2008:PBI}, 
we can conclude that the proof is complete. So, we only need to show that 
the claims (\ref{p_props}) are true. 

It is easy to see that 
\begin{align*}
&\Vert \boldsymbol{q}\Vert_{H(\text{div},\widehat{K})}\\
\leq & C (\Vert \boldsymbol{q} - \boldsymbol{p}\Vert_{L^{2}(\widehat{K})} 
+ \Vert \boldsymbol{p}\Vert_{L^{2}(\widehat{K})}
 + \Vert \widehat{\nabla}\cdot \boldsymbol{q}\Vert_{L^{2}(\widehat{K})})\\
\leq & C  (\Vert \widehat{\nabla}\cdot (\boldsymbol{q} - \boldsymbol{p})\Vert_{L^{2}(\widehat{K})} 
+ \Vert \boldsymbol{p}\Vert_{L^{2}(\widehat{K})}
 + \Vert \widehat{\nabla}\cdot \boldsymbol{q}\Vert_{L^{2}(\widehat{K})}) 
 \text{ by } \text{\cite[Lemma~5.2 case 2]{Demko:2008:PBI}}\\
 \leq & C (\Vert \widehat{\nabla}\cdot (\boldsymbol{q} - \boldsymbol{p})\Vert_{L^{2}(\widehat{K})} 
+ \Vert \boldsymbol{p}\Vert_{L^{2}(\widehat{K})}
 + \Vert \widehat{\nabla}\cdot \boldsymbol{p}\Vert_{L^{2}(\widehat{K})})
 \text{ by the definition of } \boldsymbol{q}\\
 \leq & C (\Vert \widehat{\nabla}\cdot \boldsymbol{q}\Vert_{L^{2}(\widehat{K})} 
+\Vert \widehat{\nabla}\cdot \boldsymbol{p}\Vert_{L^{2}(\widehat{K})}
+ \Vert \boldsymbol{p}\Vert_{L^{2}(\widehat{K})})\\
\leq & C (\Vert \widehat{\nabla}\cdot \boldsymbol{p}\Vert_{L^{2}(\widehat{K})}
+ \Vert \boldsymbol{p}\Vert_{L^{2}(\widehat{K})}) \text{ by the definition of } \boldsymbol{q}\\
= & C \Vert \boldsymbol{p}\Vert_{H(\text{div},\widehat{K})}.
\end{align*}
Notice that for any $\widehat{\boldsymbol{\phi}} \in \boldsymbol{Q}_{p+1,0}(\widehat{K})$, we have 
\begin{align*}
(\boldsymbol{p} + \widehat{\nabla}\times\widehat{\boldsymbol{\phi}})\cdot\widehat{\normal} 
= \boldsymbol{p}\cdot\widehat{\normal} \text{ on }\partial\widehat{K},\qquad
\widehat{\nabla}\times\widehat{\boldsymbol{\phi}}\in \boldsymbol{RT}_{p+1}(\widehat{K}).
\end{align*}
So, we have 
\begin{align*}
\Vert \boldsymbol{p} + \widehat{\nabla}\times\widehat{\boldsymbol{\phi}}\Vert_{H(\text{div})(\widehat{K})}^{2} 
= \Vert \boldsymbol{p}\Vert_{H(\text{div})(\widehat{K})}^{2} 
+ \Vert \widehat{\nabla}\times\widehat{\boldsymbol{\phi}}\Vert_{H(\text{div})(\widehat{K})}^{2}.
\end{align*}
Thus, we can conclude that $\boldsymbol{p}$ satisfies the claims (\ref{p_props}).
\end{proof}

For any $\boldsymbol{\psi} \in H^{1}(\Omega;\mathbb{R}^{d})$, 
we define $\Pi_{\boldsymbol{V}} \boldsymbol{\psi}$ by
\begin{align}
\label{proj_div_domain}
& \Pi_{\boldsymbol{V}} \boldsymbol{\psi} (G_{K}(\widehat{x})) 
= (\det DG_{K}(\widehat{x}) )^{-1} DG_{K}(\widehat{x}) (\widehat{\Pi}_{\boldsymbol{V}}\widehat{\boldsymbol{\psi}})(\widehat{x})
\quad \forall  \widehat{x}\in\widehat{K},  K\in\mathcal{T}_{h},\\
\nonumber
& \text{where } \boldsymbol{\psi}(G_{K}(\widehat{x})) =  (\det DG_{K}(\widehat{x}))^{-1} DG_{K}(\widehat{x})  
\widehat{\boldsymbol{\psi}}(\widehat{x}).
\end{align}

\begin{lemma}
\label{lemma_proj_div_domain}
Let $\boldsymbol{\psi}\in C^{\infty}(\Omega;\mathbb{R}^{d})$ satisfy
\begin{align*}
& \Vert \nabla^{p} \boldsymbol{\psi}\Vert_{L^{2}(\Omega)}
\leq \gamma^{p}\max (p, k)^{p}C_{\boldsymbol{\psi}}\quad \forall p\in \mathbb{N}_{0}.
\end{align*}
Here, $C_{\boldsymbol{\psi}}$ and $\gamma$ are independent of $k$, $h$ and $p$. 
Then, there are $C,\sigma>0$ which are also independent of $k$, $h$ and $p$, such that 
\begin{align*}
k\Vert  \Pi_{\boldsymbol{V}} \boldsymbol{\psi} - \boldsymbol{\psi}\Vert_{L^{2}(\Omega)} 
\leq C k(h(p+2))^{d}C_{\boldsymbol{\psi}}\left[\left(\dfrac{h}{h+\sigma} \right)^{p+2} 
+ \left( \dfrac{kh}{\sigma p} \right)^{p+2} \right],\\
\Vert \nabla\cdot ( \Pi_{\boldsymbol{V}} \boldsymbol{\psi} - \boldsymbol{\psi}) \Vert_{L^{2}(\Omega)} 
\leq C h^{-1}(h(p+2))^{d}C_{\boldsymbol{\psi}}\left[\left(\dfrac{h}{h+\sigma} \right)^{p+2} 
+ \left( \dfrac{kh}{\sigma p} \right)^{p+2} \right],\\
k^{1/2}\Vert  (\Pi_{\boldsymbol{V}} \boldsymbol{\psi} - \boldsymbol{\psi})\cdot\normal\Vert_{L^{2}(\partial\Omega)} 
\leq C k^{1/2}h^{-1/2}(h(p+2))^{d}C_{\boldsymbol{\psi}}\left[\left(\dfrac{h}{h+\sigma} \right)^{p+2} 
+ \left( \dfrac{kh}{\sigma p} \right)^{p+2} \right].
\end{align*}
\end{lemma}

\begin{proof}
We follow the proof of \cite[Theorem~$5.5$]{MelenkMC10}.
We start by defining for each $K\in\mathcal{T}_{h}$ the constant $C_{K}$ by
\begin{align}
\label{one_element_term}
C_{K}^{2} = \Sigma_{p\in \mathbb{N}_{0}}\dfrac{ \Vert \nabla^{p} \boldsymbol{\psi}\Vert_{L^{2}(K)}^{2}  }
{(2\gamma \max (p, k))^{2p}}.
\end{align}
It is easy to see that 
\begin{subequations}
\label{one_element_grows}
\begin{align}
\label{one_element_grow1}
& \Vert \nabla^{p} \boldsymbol{\psi}\Vert_{L^{2}(K)}\leq (2\gamma \max (p, k))^{p} 
C_{K}\quad \forall p\in\mathbb{N}_{0},\\
\label{one_element_grow2}
& \Sigma_{K\in\mathcal{T}_{h}} C_{K}^{2} \leq \frac{4C}{3} C_{\boldsymbol{\psi}}^{2}.
\end{align}
\end{subequations}
We choose $K\in\mathcal{T}_{h}$ arbitrarily. We define 
\begin{align*}
& A(\widehat{x}) = DG_{K}(\widehat{x})\quad\forall \widehat{x}\in \widehat{K},\\
& (\det A)^{-1}(\widehat{x}) A(\widehat{x}) \widehat{\boldsymbol{\psi}}(\widehat{x}) 
= \boldsymbol{\psi}(G_{K}(\widehat{x})).
\end{align*}
Let $\Pi_{\boldsymbol{V}}$ be the projection defined in (\ref{proj_div_domain}). 
Then by standard change of variable, we have 
\begin{align*}
& \Vert \Pi_{\boldsymbol{V}}\boldsymbol{\psi} - \boldsymbol{\psi} \Vert_{L^{2}(K)}  \leq C (\Vert 
\det A^{-1}\Vert_{L^{\infty}(\widehat{K})})^{1/2} \Vert A\Vert_{L^{\infty}(\widehat{K})}\Vert \widehat{\Pi}_{\boldsymbol{V}}\widehat{\boldsymbol{\psi}} - \widehat{\boldsymbol{\psi}}\Vert_{L^{2}(\widehat{K})},\\
& \Vert \nabla\cdot \Pi_{\boldsymbol{V}}\boldsymbol{\psi} - \nabla\cdot \boldsymbol{\psi}\Vert_{L^{2}(K)} 
\leq C (\Vert \det A^{-1}\Vert_{L^{\infty}(\widehat{K})})^{1/2}
\Vert \widehat{P}(\widehat{\nabla}\cdot \widehat{\boldsymbol{\psi}}) 
- \widehat{\nabla}\cdot \widehat{\boldsymbol{\psi}}\Vert_{L^{2}(\widehat{K})},\\
& \Vert (\Pi_{\boldsymbol{V}}\boldsymbol{\psi} - \boldsymbol{\psi})\cdot\normal \Vert_{L^{2}(\partial K)}  
\leq C  \Vert A\Vert_{L^{\infty}(\widehat{K})}^{(d-1)/2}\Sigma_{\widehat{F}\in \bigtriangleup_{d-1}(\widehat{K})}
\Vert \widehat{P}_{\widehat{F}}(\widehat{\boldsymbol{\psi}}\cdot\widehat{\normal}) 
- \widehat{\boldsymbol{\psi}}\cdot\widehat{\normal}\Vert_{L^{2}(\widehat{F})}.
\end{align*}
$\widehat{P}$ and $\widehat{P}_{\widehat{F}}$ are the standard $L^{2}$-orthogonal projections 
onto $P_{p+1}(\widehat{K})$ and $P_{p+1}(\widehat{F})$, respectively. The last two inequalities above 
is due to (\ref{commutativity_div_reference}) and the fact that  Piola transform commutes with 
both divergence operator and trace of normal component. Then, by the properties of matrix $A$ 
in Definition~\ref{k_regular_meshes} and Lemma~\ref{lemma_proj_div_reference}, we have 
\begin{subequations}
\label{one_element_apps}
\begin{align}
\label{one_element_app1}
& \Vert \Pi_{\boldsymbol{V}}\boldsymbol{\psi} - \boldsymbol{\psi} \Vert_{L^{2}(K)}  \leq C h^{1-d/2}\Vert \widehat{\Pi}_{\boldsymbol{V}}\widehat{\boldsymbol{\psi}} - \widehat{\boldsymbol{\psi}}\Vert_{L^{2}(\widehat{K})}\\
\nonumber
\leq & C h ^{1-d/2}\left( \inf_{\widehat{\boldsymbol{\phi}}\in \boldsymbol{RT}_{p+1}(\widehat{K})} 
\Vert \widehat{\boldsymbol{\psi}} - \widehat{\boldsymbol{\phi}} \Vert_{H(\text{div}, \hat{K})} 
+ \inf_{\widehat{\boldsymbol{\varphi}}\in \boldsymbol{RT}_{p+1}(\widehat{K})} 
\Vert (\widehat{\boldsymbol{\psi}} - \widehat{\boldsymbol{\varphi}} )\cdot\widehat{\normal} 
\Vert_{L^{2}(\partial\widehat{K})}\right)\\
\label{one_element_app2}
& \Vert \nabla\cdot \Pi_{\boldsymbol{V}}\boldsymbol{\psi} - \nabla\cdot \boldsymbol{\psi}\Vert_{L^{2}(K)} 
\leq C h^{-d/2}\Vert \widehat{P}(\widehat{\nabla}\cdot \widehat{\boldsymbol{\psi}}) 
- \widehat{\nabla}\cdot \widehat{\boldsymbol{\psi}}\Vert_{L^{2}(\widehat{K})}\\
\nonumber
\leq & C h^{-d/2}\inf_{\widehat{\boldsymbol{\phi}}\in \boldsymbol{RT}_{p+1}(\widehat{K})} 
\Vert \widehat{\nabla}\cdot (\widehat{\boldsymbol{\psi}} - \widehat{\boldsymbol{\phi}}) \Vert_{L^{2}(\hat{K})} 
= C h^{-d/2}\inf_{\widehat{v}\in \boldsymbol{P}_{p+1}(\widehat{K})} 
\Vert \widehat{\nabla}\cdot \widehat{\boldsymbol{\psi}} - \widehat{v} \Vert_{L^{2}(\hat{K})},\\
\label{one_element_app3}
& \Vert (\Pi_{\boldsymbol{V}}\boldsymbol{\psi} - \boldsymbol{\psi} ) \cdot \normal\Vert_{L^{2}(\partial K)} 
\leq C h^{(1-d)/2}\Sigma_{\widehat{F}\in \bigtriangleup_{d-1}(\widehat{K})}
\Vert\widehat{P}_{\widehat{F}}(\widehat{\boldsymbol{\psi}}\cdot\widehat{\normal})
- \widehat{\boldsymbol{\psi}}\cdot\boldsymbol{\normal}\Vert_{L^{2}(\widehat{F})}\\
\nonumber
\leq & C h^{(1-d)/2}\Sigma_{\widehat{F}\in\bigtriangleup_{d-1}(\widehat{K})}
\inf_{\widehat{v}\in \boldsymbol{P}_{p+1}(\widehat{F})} 
\Vert \widehat{\boldsymbol{\psi}}\cdot \widehat{\normal} - \widehat{v} \Vert_{L^{2}(\widehat{F})}.
\end{align}
\end{subequations}

By the definition of $\widehat{\boldsymbol{\psi}}$, we have 
\begin{align*}
\widehat{\boldsymbol{\psi}}(\widehat{x}) = \text{adj} A(\widehat{x})\boldsymbol{\psi}(G_{K}(\widehat{x}))
\quad \forall \widehat{x}\in \widehat{K}.
\end{align*}
Here, $\text{adj}A$ is the adjoint matrix of $A$. By the properties of matrix $A$ in Definition~\ref{k_regular_meshes} 
and \cite[Lemma~$A.1.3$]{MelenkBook02}, we have 
\begin{align}
\label{adj_app}
\Vert \widehat{\nabla}^{p}\text{adj}A\Vert_{L^{\infty}(\widehat{K})}
\leq C h^{p+d-1}\gamma^{p+d-1} (p+d-1)!\quad \forall p\in \mathbb{N}_{0}.
\end{align}
By (\ref{one_element_grow1}), the properties of matrix $A$ in Definition~\ref{k_regular_meshes} 
and \cite[Lemma~$C.1$]{MelenkMC10}, we have 
\begin{align*}
\Vert \widehat{\nabla}^{p}(\boldsymbol{\psi}\circ G_{K}) \Vert_{L^{2}(\widehat{K})} 
\leq C h^{p+d/2}\gamma_{1}^{p} \max (p, k)^{p} C_{K}\quad \forall p\in \mathbb{N}_{0}. 
\end{align*}
Here, $\gamma_{1}>0$ is independent of $h$, $k$ and $p$. Then, by combining the above inequality with (\ref{adj_app}) 
and applying \cite[Lemma~$A.1.3$]{MelenkBook02} again, we have 
\begin{align}
\label{semi_norm_p}
\Vert \widehat{\nabla}^{p} \widehat{\boldsymbol{\psi}}\Vert_{L^{2}(\widehat{K})} 
\leq C h^{3d/2-1}(p\cdots (p+d-1))(h\gamma)^{p}\max (p, k)^{p} C_{K}.
\end{align}
Here, the constant $C$ is independent of $h$, $k$ and $p$. Now, we can apply \cite[Lemma~$C.2$]{MelenkMC10} 
with $R=1$ and $C_{u} = C h^{3d/2-1}(p\cdots (p+d-1)) C_{K}$, such that 
\begin{align}
\label{infty_norm_app}
& \inf_{\widehat{\boldsymbol{\phi}}\in P_{p+1}(\widehat{K};\mathbb{R}^{d})}  
\Vert \widehat{\boldsymbol{\psi}} - \widehat{\boldsymbol{\phi}} \Vert_{W^{1,\infty}(\hat{K})} \\
\nonumber
\leq & C h^{3d/2-1}(p\cdots (p+d-1)) C_{K}\left[\left(\dfrac{h}{h+\sigma} \right)^{p+2} 
+ \left( \dfrac{kh}{\sigma p} \right)^{p+2} \right]\\
\nonumber
\leq & C h^{3d/2-1}(p+2)^{d} C_{K}\left[\left(\dfrac{h}{h+\sigma} \right)^{p+2} 
+ \left( \dfrac{kh}{\sigma p} \right)^{p+2} \right].
\end{align}

According to (\ref{one_element_app1}) and (\ref{infty_norm_app}), we have 
\begin{align*}
\Vert \Pi_{\boldsymbol{V}}\boldsymbol{\psi} - \boldsymbol{\psi} \Vert_{L^{2}(K)} 
\leq C (h(p+2))^{d}C_{K}\left[\left(\dfrac{h}{h+\sigma} \right)^{p+2} 
+ \left( \dfrac{kh}{\sigma p} \right)^{p+2} \right].
\end{align*}
Then, by (\ref{one_element_grow2}), we have 
\begin{align*}
\Vert  \Pi_{\boldsymbol{V}} \boldsymbol{\psi} - \boldsymbol{\psi}\Vert_{L^{2}(\Omega)} 
\leq C (h(p+2))^{d}C_{\boldsymbol{\psi}}\left[\left(\dfrac{h}{h+\sigma} \right)^{p+2} 
+ \left( \dfrac{kh}{\sigma p} \right)^{p+2} \right].
\end{align*}
We notice that $\widehat{\nabla}\cdot \widehat{\boldsymbol{\psi}} = \det A \nabla\cdot\boldsymbol{\psi}$. 
Then, by similar argument, we have the estimate of 
$\Vert \nabla\cdot ( \Pi_{\boldsymbol{V}} \boldsymbol{\psi} - \boldsymbol{\psi}) \Vert_{L^{2}(\Omega)}$. 
By (\ref{one_element_app3}) and (\ref{infty_norm_app}), we have the estimate of 
$\Vert  (\Pi_{\boldsymbol{V}} \boldsymbol{\psi} - \boldsymbol{\psi})\cdot\normal\Vert_{L^{2}(\partial\Omega)} $.
So, we can conclude that the proof is complete.
\end{proof}

\begin{lemma}
\label{lemma_div_h2}
Let $\boldsymbol{\psi}\in \{ \boldsymbol{\varphi}\in H^{1}(\Omega; \mathbb{R}^{d}): 
\nabla\cdot \boldsymbol{\varphi}\in H^{1}(\Omega) \}$ satisfy
\begin{align*}
k^{2}\Vert \boldsymbol{\psi}\Vert_{L^{2}(\Omega)} + k \Vert \boldsymbol{\psi}\Vert_{H^{1}(\Omega)} 
+ \Vert \nabla\cdot \boldsymbol{\psi}\Vert_{H^{1}(\Omega)} \leq C C_{\boldsymbol{\psi}}. 
\end{align*}
Then, there exists $\boldsymbol{\psi}_{h}\in\boldsymbol{V}_{h}$ such that  
\begin{align*}
&  k\Vert \boldsymbol{\psi} 
- \boldsymbol{\psi}_{h}\Vert_{L^{2}(\Omega)}
+\Vert \nabla\cdot ( \boldsymbol{\psi} - \boldsymbol{\psi}_{h} )\Vert_{L^{2}(\Omega)} 
\leq C h C_{\boldsymbol{\psi}},\\
& k^{1/2}\Vert (\boldsymbol{\psi} 
- \boldsymbol{\psi}_{h})\cdot\normal\Vert_{L^{2}(\partial\Omega)} 
\leq C h^{1/2}k^{-1/2} C_{\boldsymbol{\psi}}.
\end{align*}
\end{lemma}

\begin{proof}
Let $\widehat{\Pi}_{\boldsymbol{RT}}^{0}$ be the lowest order standard Raviart-Thomas projection on $\widehat{K}$. 
We notice that for any $\widehat{\boldsymbol{\psi}} \in H^{1}(\widehat{K};\mathbb{R}^{d})$
\begin{align*}
\widehat{\nabla}\cdot \widehat{\Pi}_{\boldsymbol{RT}}^{0}\widehat{\boldsymbol{\psi}} 
= \widehat{P}_{0} \widehat{\nabla}\cdot \widehat{\boldsymbol{\psi}},\qquad
(\widehat{\Pi}_{\boldsymbol{RT}}^{0}\widehat{\boldsymbol{\psi}}) \cdot \widehat{\normal} 
= \widehat{P}_{0,\widehat{F}} (\widehat{\boldsymbol{\psi}}\cdot\boldsymbol{\normal}) 
\quad\forall \widehat{F}\in\bigtriangleup_{d-1}(\widehat{K}). 
\end{align*}
Here, $\widehat{P}_{0}$ and $\widehat{P}_{0,\widehat{F}}$ are standard $L^{2}$-orthogonal projection 
onto $P_{0}(\widehat{K})$ and $P_{0}(\widehat{F})$, respectively. 
Then, we define $\Pi_{RT}^{0}$ in the same way as $\Pi_{\boldsymbol{V}}$ in (\ref{proj_div_domain}) except that 
we replace $\widehat{\Pi}_{\boldsymbol{V}}$ by $\widehat{\Pi}_{\boldsymbol{RT}}^{0}$. According to the fact that 
the Piola transform commutes with both divergence operator and trace of normal component, we immediately have 
\begin{align*}
& k\Vert \Pi_{RT}^{0}\boldsymbol{\psi} - \boldsymbol{\psi} \Vert_{L^{2}(\Omega)} 
+ \Vert \nabla\cdot (\Pi_{RT}^{0}\boldsymbol{\psi} - \boldsymbol{\psi}) 
\Vert_{L^{2}(\Omega)} \leq C h C_{\boldsymbol{\psi}},\\
& k^{1/2}\Vert (\Pi_{RT}^{0}\boldsymbol{\psi} - \boldsymbol{\psi})\cdot\normal \Vert_{L^{2}(\partial \Omega)} 
\leq C h^{1/2} k ^{-1/2} C_{\boldsymbol{\psi}}.
\end{align*}
So, we can conclude that the proof is complete.
\end{proof}

\begin{remark}
We notice that in \cite[Theorem~$3.5$]{Monk92}, it is shown that on a reference cube, the standard 
$H(\text{div})$-conforming {\Nedelec} projection $\widehat{\pi}_{D}$ has the following approximation property
\begin{align}
\label{rt_proj_cube}
\Vert \widehat{\pi}_{D} \boldsymbol{\psi} - \boldsymbol{\psi} \Vert_{L^{2}(\widehat{K})} 
\leq C p^{-1/2} \Vert \boldsymbol{\psi}\Vert_{H^{1}(\widehat{K})}.
\end{align}
However, there is no $H(\text{div})$-conforming projection on triangle or tetrahedron elements satisfying (\ref{rt_proj_cube}).
This is the reason why we use the lowest order standard Raviart-Thomas projection on $\widehat{K}$. If there is 
a $H(\text{div})$-conforming projection on triangular meshes satisfying (\ref{rt_proj_cube}), then the convergent result 
of Theorem~\ref{main_res} can be improved as 
\begin{align*}
& \Vert u-u_{h} \Vert_{L^{2}(\Omega)} \\
\leq & C h p^{-1/2} \left( k\Vert \boldsymbol{\phi} - \boldsymbol{\psi}_{h}\Vert_{L^{2}(\Omega)} 
+ \Vert\nabla\cdot (\boldsymbol{\phi} - \boldsymbol{\psi}_{h}) \Vert_{L^{2}(\Omega)} 
+ \Vert \nabla (u - v_{h})\Vert_{L^{2}(\Omega)} 
+ k \Vert u - v_{h}\Vert_{L^{2}(\Omega)} \right)\\
& \quad + C h^{1/2}p^{-1/2} \Vert (\boldsymbol{\phi} - \boldsymbol{\psi}_{h})\cdot\normal\Vert_{L^{2}(\partial\Omega)},
\end{align*}
for any $(\boldsymbol{\psi}_{h}, v_{h})\in \boldsymbol{V}_{h}\times W_{h}$.
\end{remark}

We define two projections $\Pi_{W}, \tilde{\Pi}_{W}: H^{2}(\Omega)\rightarrow W_{h}$ by
\begin{align}
\label{h1_projs}
\left( \Pi_{W} v \right)|_{K} = \pi \left( v \circ G_{K}^{-1}\right)\quad 
\left( \tilde{\Pi}_{W} v \right)|_{K} = \tilde{\pi} \left( v \circ G_{K}^{-1}\right)
\quad \forall K\in\mathcal{T}_{h},  
\end{align}
where $\pi$ and $\tilde{\pi}$ are the projections from $H^{2}(\widehat{K})$ onto $P_{p+1}(\widehat{K})$ 
introduced in \cite[Theorem~$B.4$, Lemma~$C.3$]{MelenkMC10}, respectively. 

\begin{lemma}
\label{lemma_h1_app}
Let $v_{H^{2}}\in H^{2}(\Omega)$ and $v_{A}$ is an analytic function in $\Omega$.
We assume that there are constants $C_{v},\gamma>0$ such that 
\begin{align*}
\Vert v_{H^{2}} \Vert_{H^{2}(\Omega)} \leq C_{v} \quad 
\Vert \nabla^{p} v_{A}\Vert_{L^{2}(\Omega)} \leq \gamma^{p} \max (p, k)^{p} C_{v}.
\end{align*}
Then, there are constants $C,\sigma>0$ independent of $h,k$ and $p$ such that 
\begin{align*}
& k\Vert \tilde{\Pi}_{W} v_{A} - v_{A} \Vert_{L^{2}(\Omega)} + \Vert  \tilde{\Pi}_{W} v_{A} - v_{A}\Vert_{H^{1}(\Omega)}
+ k^{1/2}\Vert  \tilde{\Pi}_{W} v_{A} - v_{A}\Vert_{H^{1/2}(\partial\Omega)} \\
\leq & C \left[ \left( \dfrac{h}{h+\sigma} \right)^{p}\left(1 + \dfrac{hk}{h+\sigma} \right) 
+ k\left( \dfrac{kh}{\sigma p}\right)^{p}\left(\frac{1}{p} + \dfrac{kh}{\sigma p} \right)  \right]C_{v},\\
& k\Vert \Pi_{W} v_{H^{2}} - v_{H^{2}} \Vert_{L^{2}(\Omega)} + \Vert  \Pi_{W} v_{H^{2}} - v_{H^{2}}\Vert_{H^{1}(\Omega)}
+ k^{1/2}\Vert  \Pi_{W} v_{H^{2}} - v_{H^{2}}\Vert_{H^{1/2}(\partial\Omega)}\\
\leq & C k^{-1}\left( \dfrac{kh}{p} +\left( \dfrac{kh}{p}\right)^{2} \right)C_{v}.
\end{align*}
\end{lemma}

\begin{proof}
The proof is a simple consequence of the procedure in \cite[Theorem~$5.5$]{MelenkMC10}. 
We have used the fact that for any $k\geq 1$, 
\begin{align*}
k^{1/2} \Vert v\Vert_{L^{2}(\partial\Omega)} \leq C\left( k \Vert v\Vert_{L^{2}(\Omega)} + 
\Vert v\Vert_{H^{1}(\Omega)} \right)\quad \forall v\in H^{1}(\Omega).
\end{align*}
\end{proof}

\section{Duality argument}

We recall (\ref{high_der}) that 
$\vert \nabla^{n} u (x)\vert^{2} = \Sigma_{\alpha\in \mathbb{N}_{0}^{d}:\vert \alpha\vert = n} \dfrac{n!}{\alpha !}
\vert D^{\alpha} u (x)\vert^{2}$. 

\begin{lemma}
\label{lemma_test_function}
We assume that the Assumptions (A1, A2) hold. 
Then, for any $(\boldsymbol{\varphi}, w)\in \boldsymbol{V}\times W$, there is 
$(\boldsymbol{\psi}, v)\in \boldsymbol{V}\times W$ such that $\Vert w\Vert_{L^{2}(\Omega)}^{2}
= b((\boldsymbol{\varphi}, w), (\boldsymbol{\psi}, v))$. 

$(\boldsymbol{\psi}, v)$ can be written as 
$(\boldsymbol{\psi}, v) = (\boldsymbol{\psi}_{A}, v_{A}) + (\boldsymbol{\psi}_{H^{2}}, v_{H^{2}})$. 
Here, both $\boldsymbol{\psi}_{A}$ and $v_{A}$ are analytic functions in $\Omega$, 
$\boldsymbol{\psi}_{H^{2}}\in H^{1}(\Omega;\mathbb{R}^{d})$, and $v_{H^{2}}\in H^{2}(\Omega)$. 
There are constants $C,\gamma>0$ independent of $k\geq k_{0}$ such that 
\begin{subequations}
\label{test_func_ineqs}
\begin{align}
\label{test_func_ineq1}
& k\Vert \boldsymbol{\psi}_{A}\Vert_{L^{2}(\Omega)} + \Vert \boldsymbol{\psi}_{A}\Vert_{H^{1}(\Omega)} 
\leq C k \Vert w\Vert_{L^{2}(\Omega)},\\
\label{test_func_ineq2}
& k\Vert v_{A}\Vert_{L^{2}(\Omega)} + \Vert v_{A}\Vert_{H^{1}(\Omega)}
\leq C k \Vert w\Vert_{L^{2}(\Omega)},\\
\label{test_func_ineq3}
& \Vert \nabla^{p+2} \boldsymbol{\psi}_{A}\Vert_{L^{2}(\Omega)} + \Vert \nabla^{p+2} v_{A}\Vert_{L^{2}(\Omega)}
\leq C\gamma^{p}\max (p, k)^{p+2}\Vert w\Vert_{L^{2}(\Omega)},\\
\label{test_func_ineq4}
& k^{2}\Vert \boldsymbol{\psi}_{H^{2}}\Vert_{L^{2}(\Omega)} + k \Vert \boldsymbol{\psi}_{H^{2}}\Vert_{H^{1}(\Omega)} 
\leq C \Vert w\Vert_{L^{2}(\Omega)},\\
\label{test_func_ineq5}
&  \Vert \nabla\cdot \boldsymbol{\psi}_{H^{2}}\Vert_{H^{1}(\Omega)} \leq C \Vert w\Vert_{L^{2}(\Omega)},\\
\label{test_func_ineq6}
& k^{2}\Vert v_{H^{2}} \Vert_{L^{2}(\Omega)} + k\Vert v_{H^{2}}\Vert_{H^{1}(\Omega)} +\Vert v_{H^{2}}\Vert_{H^{2}(\Omega)} 
\leq C \Vert w\Vert_{L^{2}(\Omega)}.
\end{align}
\end{subequations}
\end{lemma}

\begin{remark}
\label{remark_rt}
Since we can only have $\Vert \nabla\cdot \boldsymbol{\psi}_{H^{2}}\Vert_{H^{1}(\Omega)} \leq C \Vert w\Vert_{L^{2}(\Omega)}$ 
instead of $\Vert \boldsymbol{\psi}_{H^{2}}\Vert_{H^{2}(\Omega)} \leq C \Vert w\Vert_{L^{2}(\Omega)}$, it is necessary to use 
Raviart-Thomas space to approximate $\boldsymbol{\phi}$ and $\boldsymbol{\psi}$ instead of vector valued continuous 
piece-wise polynomial space, in order to show quasi optimal convergence.
\end{remark}

\begin{remark}
\label{remark_weight}
For any $(\boldsymbol{\phi}, u), (\boldsymbol{\psi}, v)\in \boldsymbol{V}\times W$, we define 
\begin{align*}
& b_{\tau}((\boldsymbol{\phi}, u), (\boldsymbol{\psi}, v)) \\
\nonumber
= & (\im k \boldsymbol{\phi}+\nabla u, \im k \boldsymbol{\psi} + \nabla v)_{\Omega} 
+ (\im k u + \nabla\cdot \boldsymbol{\phi}, \im k v + \nabla\cdot \boldsymbol{\psi})_{\Omega}
 + \tau \langle \boldsymbol{\phi}\cdot\normal + u, \boldsymbol{\psi}\cdot\normal + v \rangle_{\partial\Omega},
\end{align*}
where $\tau$ is a positive constant. It is easy to see that the exact solution 
$(\boldsymbol{\phi}= \im k^{-1}\nabla u , u)$ satisfies
\begin{align*}
b_{\tau}((\boldsymbol{\phi}, u), (\boldsymbol{\psi}, v)) = 
 (-\im fk^{-1}, \im k v+ \nabla\cdot \boldsymbol{\psi})_{\Omega} 
+ \tau k^{-1}\langle \im g,  \boldsymbol{\psi}\cdot\normal + v\rangle_{\partial \Omega}
\quad\forall (\boldsymbol{\psi},v)\in \boldsymbol{V}\times W.
\end{align*} 
As we mentioned in Remark~\ref{remark_stability}, the variational form $b_{\tau}$ is uniformly coercive with respect to 
the wave number $k$ if $\tau\geq 1$. However, if we choose $\tau$ to be $1$, then the boundary condition 
(\ref{duality_two_eq3}) in the proof of Lemma~\ref{lemma_test_function} will be 
\begin{align*}
\boldsymbol{\phi}\cdot\normal + v = \im k z\quad \text{ on }\partial\Omega.
\end{align*}
The consequence is that all right hand sides of regularity estimates (\ref{test_func_ineqs}) have to be multiplied 
by an extra factor $k^{1/2}$, such that the quasi optimal convergent result in Theorem~\ref{main_res} 
can {\em not } be obtained.
\end{remark}

\begin{proof}
Let $z$ be the solution of the Helmholtz equation
\begin{align*}
-\Delta z - k^{2}z = w\text{ in }\Omega,\qquad
\dfrac{\partial z}{\partial\normal} + \im k z = 0\text{ on }\partial\Omega.
\end{align*}
It is easy to check that 
\begin{align*}
& \Vert w\Vert_{L^{2}(\Omega)}^{2} \\
= & (\nabla w, \nabla z)_{\Omega} 
-k^{2}(w, z)_{\Omega} -\im k \langle w, z\rangle_{\partial\Omega}\\
= & (\im k \boldsymbol{\varphi}+\nabla w, \nabla z)_{\Omega}
-(\im k \boldsymbol{\varphi}, \nabla z)_{\Omega} - k^{2}(w, z)_{\Omega} 
-\im k \langle w, z \rangle_{\partial \Omega}\\
= & (\im k \boldsymbol{\varphi}+\nabla w, \nabla z)_{\Omega} 
+ (\nabla\cdot \boldsymbol{\varphi}+\im k w, -\im kz)_{\Omega}
+\langle \boldsymbol{\varphi}\cdot\normal + w, \im k z \rangle_{\partial\Omega}.
\end{align*}

According to Theorem~\ref{thm_divergence_regularity} and Assumption (A2),  
$z$ can be written as $z = z_{A} + z_{H^{2}}$. 
$z_{A}$ is an analytic function and $z_{H^{2}}\in H^{2}(\Omega)$. In addition, 
\begin{subequations}
\label{duality_one_ineqs}
\begin{align}
\label{duality_one_ineq1}
& k\Vert z_{A}\Vert_{L^{2}(\Omega)} + \Vert z_{A}\Vert_{H^{1}(\Omega)}
\leq C\Vert w\Vert_{L^{2}(\Omega)},\\
\label{duality_one_ineq2}
& \Vert \nabla^{p+2} z_{A}\Vert_{L^{2}(\Omega)}
\leq C\gamma^{p}k^{-1}\max (p, k)^{p+2}\Vert w\Vert_{L^{2}(\Omega)},\\
\label{duality_one_ineq3}
&\Vert z_{H^{2}}\Vert_{H^{2}(\Omega)} + k\Vert z_{H^{2}}\Vert_{H^{1}(\Omega)} 
+ k^{2}\Vert z_{H^{2}}\Vert_{L^{2}(\Omega)}\leq C \Vert w\Vert_{L^{2}(\Omega)}.
\end{align}
\end{subequations}

According to Theorem~\ref{thm_stability}, we can define $(\boldsymbol{\psi}, v)\in\boldsymbol{V}\times W$ by
\begin{subequations}
\label{duality_two_eqs}
\begin{align}
\label{duality_two_eq1}
& \im k \boldsymbol{\psi} + \nabla v = \nabla z \quad \text{ in }\Omega,\\
\label{duality_two_eq2}
& \im k v + \nabla\cdot \boldsymbol{\psi} = -\im k z \quad \text{ in }\Omega,\\
\label{duality_two_eq3}
& k^{1/2}\left( \boldsymbol{\psi}\cdot\normal + v \right) = \im k^{1/2} z \quad \text{ on }\partial\Omega.
\end{align}
\end{subequations}

Obviously, we can write $(\boldsymbol{\psi}, v)$ as $(\tilde{\boldsymbol{\psi}}_{A}, \tilde{v}_{A}) 
+ (\tilde{\boldsymbol{\psi}}_{H^{2}}, \tilde{v}_{H^{2}})$ where 
\begin{align*}
& \im k \tilde{\boldsymbol{\psi}}_{A} + \nabla \tilde{v}_{A} = \nabla z_{A} 
 & \im k \tilde{\boldsymbol{\psi}}_{H^{2}} + \nabla \tilde{v}_{H^{2}} = \nabla z_{H^{2}}
\quad \text{ in }\Omega,\\
& \im k \tilde{v}_{A} + \nabla\cdot \tilde{\boldsymbol{\psi}}_{A} = -\im k z_{A} 
 & \im k \tilde{v}_{H^{2}} + \nabla\cdot \tilde{\boldsymbol{\psi}}_{H^{2}} = -\im k z_{H^{2}}
\quad \text{ in }\Omega,\\
& k^{1/2}\left( \tilde{\boldsymbol{\psi}}_{A}\cdot\normal + \tilde{v}_{A} \right) = \im k^{1/2} z_{A} 
 & k^{1/2}\left( \tilde{\boldsymbol{\psi}}_{H^{2}}\cdot\normal + \tilde{v}_{H^{2}} \right) = \im k^{1/2} z_{H^{2}}
\quad \text{ on }\partial\Omega.
\end{align*}
It is easy to see that 
\begin{align*}
& \Delta \tilde{v}_{A} + k^{2} \tilde{v}_{A} = \Delta z_{A} - k^{2}z_{A}
& \Delta \tilde{v}_{H^{2}} + k^{2} \tilde{v}_{H^{2}} = \Delta z_{H^{2}} - k^{2}z_{H^{2}}\quad \text{ in }\Omega,\\
& \dfrac{\partial \tilde{v}_{A}}{\partial\normal} - \im k \tilde{v}_{A} = k z_{A} + \dfrac{\partial z_{A}}{\partial\normal}
& \dfrac{\partial \tilde{v}_{H^{2}}}{\partial\normal} - \im k \tilde{v}_{H^{2}} = k z_{H^{2}} 
+\dfrac{\partial z_{H^{2}}}{\partial\normal} \quad \text{ on }\partial\Omega.
\end{align*}

It is easy to see that $\tilde{v}_{A} - z_{A} = S_{k} (-2k^{2}z_{A}, (1+\im)k z_{A})$. By Lemma~\ref{lemma_analytic}, we have 
\begin{align*}
& \Vert \nabla^{p+2} \tilde{v}_{A}\Vert_{L^{2}(\Omega)}
\leq \Vert \nabla^{p+2}(\tilde{v}_{A} - z_{A})\Vert_{L^{2}(\Omega)}
+ \Vert \nabla^{p+2}z_{A}\Vert_{L^{2}(\Omega)} \\
\leq & C \gamma^{p+2}  \max (p+2, k)^{p+2}
\left( k\Vert z_{A} \Vert_{L^{2}(\Omega)} + \Vert \nabla z_{A} \Vert_{L^{2}(\Omega)}\right)
+ \Vert \nabla^{p+2}z_{A}\Vert_{L^{2}(\Omega)} \\
\leq & C \gamma^{p}  \max (p, k)^{p+2} \Vert w\Vert_{L^{2}(\Omega)}, \\
& k\Vert \tilde{v}_{A}\Vert_{L^{2}(\Omega)} + \Vert \tilde{v}_{A}\Vert_{H^{1}(\Omega)} 
\leq C k\Vert w\Vert_{L^{2}(\Omega)}.
\end{align*}
Since $\im k \tilde{\boldsymbol{\psi}}_{A} + \nabla \tilde{v}_{A} = \nabla z_{A}$, we have 
\begin{align*}
& \Vert \nabla^{p+2} \tilde{\boldsymbol{\psi}}_{A}\Vert_{L^{2}(\Omega)}
\leq C \gamma^{p}  \max (p, k)^{p+2} \Vert w\Vert_{L^{2}(\Omega)}, \\
& k\Vert \tilde{\boldsymbol{\psi}}_{A}\Vert_{L^{2}(\Omega)} 
+ \Vert \tilde{\boldsymbol{\psi}}_{A}\Vert_{H^{1}(\Omega)} 
\leq C k\Vert w\Vert_{L^{2}(\Omega)}.
\end{align*}

By (\ref{duality_one_ineq3}), we have 
\begin{align*}
\Vert \Delta z_{H^{2}} - k^{2}z_{H^{2}}\Vert_{L^{2}(\Omega)} 
+ \Vert k z_{H^{2}} + \dfrac{\partial z_{H^{2}}}{\partial\normal}\Vert_{H^{1/2}(\partial\Omega)} 
\leq C \Vert w\Vert_{L^{2}(\Omega)}.
\end{align*}
By Theorem~\ref{thm_divergence_regularity} and the Assumption (A2) again, 
$\tilde{v}_{H^{2}}$ can be written as 
$\tilde{v}_{H^{2}}= \tilde{v}_{A,H^{2}}+\tilde{v}_{H^{2},H^{2}}$ where $\tilde{v}_{A,H^{2}}$ 
and $\tilde{v}_{H^{2},H^{2}}$ are analytic and in $H^{2}(\Omega)$, respectively. In addition, we have 
\begin{align*}
k^{2}\Vert \tilde{v}_{H^{2},H^{2}}\Vert_{L^{2}(\Omega)}+k\Vert \tilde{v}_{H^{2},H^{2}}\Vert_{H^{1}(\Omega)} 
+\Vert \tilde{v}_{H^{2},H^{2}}\Vert_{H^{2}(\Omega)}\leq C \Vert w\Vert_{L^{2}(\Omega)}.
\end{align*}
We define $v_{A} = \tilde{v}_{A} + \tilde{v}_{A,H^{2}}$ and $v_{H^{2}} = \tilde{v}_{H^{2},H^{2}}$. 
Then, (\ref{test_func_ineq2}, \ref{test_func_ineq6}) hold.
We write $\tilde{\boldsymbol{\psi}}_{H^{2}}$ as $\tilde{\boldsymbol{\psi}}_{H^{2}}
=\tilde{\boldsymbol{\psi}}_{A,H^{2}}+\tilde{\boldsymbol{\psi}}_{H^{2},H^{2}}$ such that 
$\im k \tilde{\boldsymbol{\psi}}_{A,H^{2}} + \nabla \tilde{v}_{A,H^{2}}=0$. 
We define $\boldsymbol{\psi}_{A} = \tilde{\boldsymbol{\psi}}_{A} + \tilde{\boldsymbol{\psi}}_{A,H^{2}}$. 
Then, (\ref{test_func_ineq1}, \ref{test_func_ineq3}, \ref{test_func_ineq4}) hold. 

Now, we only need to prove (\ref{test_func_ineq5}). 
Since $\im k\tilde{\boldsymbol{\psi}}_{A,H^{2}}+\nabla\tilde{v}_{A,H^{2}}=0$ and (\ref{decomp_ineq3}), we have 
\begin{align*}
& \Vert \im k \tilde{v}_{A,H^{2}} + \nabla\cdot \tilde{\boldsymbol{\psi}}_{A,H^{2}}\Vert_{H^{1}(\Omega)} 
= k^{-1}\Vert k^{2} \tilde{v}_{A,H^{2}} + \Delta \tilde{v}_{A,H^{2}}\Vert_{H^{1}(\Omega)} \\
\leq & C \left(\Vert \Delta z_{H^{2}} - k^{2}z_{H^{2}}\Vert_{L^{2}(\Omega)} 
+ \Vert k z_{H^{2}} + \dfrac{\partial z_{H^{2}}}{\partial\normal}\Vert_{H^{1/2}(\partial\Omega)} \right)
\leq C \Vert w\Vert_{L^{2}(\Omega)}.
\end{align*}
Notice that 
\begin{align*}
& (\im k \tilde{v}_{H^{2},H^{2}} + \nabla\cdot \tilde{\boldsymbol{\psi}}_{H^{2},H^{2}}) 
+ (\im k \tilde{v}_{A,H^{2}} + \nabla\cdot \tilde{\boldsymbol{\psi}}_{A,H^{2}})\\ 
= & \im k \tilde{v}_{H^{2}} + \nabla\cdot \tilde{\boldsymbol{\psi}}_{H^{2}} =-\im k z_{H^{2}}.
\end{align*} 
We define $\boldsymbol{\psi}_{H^{2}} = \tilde{\boldsymbol{\psi}}_{H^{2}, H^{2}}$.
Then, we can conclude that (\ref{test_func_ineq5}) is true. 
\end{proof}

Now we can provide the proof for the quasi optimal convergent result in Theorem~\ref{main_res}.
\begin{proof}(Theorem \ref{main_res})
We denote by $e^{\boldsymbol{\phi}} = \boldsymbol{\phi} - \boldsymbol{\phi}_{h}$ and $e^{u} = u - u_{h}$.
Applying Lemma~\ref{lemma_test_function} with $w = e^{u}$, we have 
\begin{align*}
& \Vert e^{u}\Vert_{L^{2}(\Omega)}^{2} = b(( e^{\boldsymbol{\phi}}, e^{u}), (\boldsymbol{\psi}, v)) \\
= & b((e^{\boldsymbol{\phi}}, e^{u}), (\boldsymbol{\psi}_{A}, v_{A})) 
+ b((e^{\boldsymbol{\phi}}, e^{u}), (\boldsymbol{\psi}_{H^{2}}, v_{H^{2}})),
\end{align*}
such that $\boldsymbol{\psi}_{A}, v_{A}, \boldsymbol{\psi}_{H^{2}}, v_{H^{2}}$ satisfy (\ref{test_func_ineqs}). Then, 
by the standard Galerkin orthogonality of the first order system least squares method (\ref{fosls_hwn}), we have 
that for any $\widetilde{\boldsymbol{\psi}}_{A},\widetilde{\boldsymbol{\psi}}_{H^{2}} \in \boldsymbol{V}_{h}$ and 
any $\widetilde{v}_{A}, \widetilde{v}_{H^{2}} \in W_{h}$
\begin{align*}
\Vert e^{u}\Vert_{L^{2}(\Omega)}^{2} =b((e^{\boldsymbol{\phi}}, e^{u}), 
(\boldsymbol{\psi}_{A}-\widetilde{\boldsymbol{\psi}}_{A}, v_{A}-\widetilde{v}_{A})) 
+ b((e^{\boldsymbol{\phi}}, e^{u}), (\boldsymbol{\psi}_{H^{2}}-\widetilde{\boldsymbol{\psi}}_{H^{2}}, 
v_{H^{2}}-\widetilde{v}_{H^{2}})).
\end{align*}

We choose $\widetilde{\boldsymbol{\psi}}_{A}=\Pi_{\boldsymbol{V}}\boldsymbol{\psi}_{A}$ 
($\Pi_{\boldsymbol{V}}$ defined in (\ref{proj_div_domain})),  
$\widetilde{\boldsymbol{\psi}}_{H^{2}}$ to be $\boldsymbol{\psi}_{h}$ in Lemma~\ref{lemma_div_h2}, 
$\widetilde{v}_{A} = \tilde{\Pi}_{W} v_{A}$ and $\widetilde{v}_{H^{2}}= \Pi_{W}v_{H^{2}}$ 
($\tilde{\Pi}_{W}$ and $\Pi_{W}$ defined in (\ref{h1_projs})).
Then by Lemma~\ref{lemma_proj_div_domain}, Lemma~\ref{lemma_div_h2} and Lemma~\ref{lemma_h1_app}, 
it is straightforward to see that if (\ref{p_condition}) holds,
\begin{align*}
& \Vert e^{u}\Vert_{L^{2}(\Omega)}\\
\leq & C h \left( \Vert \im k e^{\boldsymbol{\phi}} + \nabla e^{u} \Vert_{L^{2}(\Omega)}  
+ \Vert \im k e^{u} + \nabla\cdot e^{\boldsymbol{\phi}}\Vert_{L^{2}(\Omega)}\right)\\ 
&\quad + C h^{1/2} \Vert e^{\boldsymbol{\phi}}\cdot\normal + e^{u} \Vert_{L^{2}(\partial\Omega)}.
\end{align*}
Finally, by Theorem~\ref{thm_stability} (the stability estimate) and the standard Galerkin orthogonality argument 
for the first order system least squares method (\ref{fosls_hwn}) again, we can conclude that the proof is complete.
\end{proof}

\section{Numerical results}
In this section, we present some numerical
results of the FOSLS method for the following 2-d Helmholtz problem:
\begin{eqnarray*}
&-\Delta u - \kappa^2 u = f := \frac{\sin{\kappa r}}{r} \qquad &{\rm
in} \ \Omega, \\
&\frac{\partial u}{\partial n} + \textbf{i} \kappa u = g \qquad
&{\rm on} \ \partial \Omega.
\end{eqnarray*}
Here $\Omega$ is unit square $[-0.5,0.5]\times[-0.5,0.5]$, and $g$ is chosen such that the exact solution is given by
\begin{eqnarray*}
u = \frac{\cos{\kappa r}}{\kappa} - \frac{\cos{\kappa} +
\textbf{i}\sin{\kappa}}{\kappa ( J_0(\kappa) + \textbf{i}J_1(\kappa)
)} J_0(\kappa r)
\end{eqnarray*}
in polar coordinates, where $J_\nu (z)$ are Bessel functions of the first kind.

In the following experiments, the FOSLS method is implemented for the pair of finite element spaces $\boldsymbol{V}_{h}  \times W_{h} $ with $p+1=1,2,3,4$, which are denoted by $RT1P1, RT2P2, RT3P3$ and $RT4P4$. For the fixed wave number $k$, we first show the dependence of  the relative error $\|u-u_h\|_{L^2(\Omega)}/\|u\|_{L^2(\Omega)}$ and  $\|\boldsymbol{\phi}-\boldsymbol{\phi}_h\|_{L^2(\Omega)}/\|\boldsymbol{\phi}\|_{L^2(\Omega)}$ on polynomial degree $p$ and mesh size $h$. The Figure  \ref{fig-1}  displaying the relative error is plotted in log-log coordinates. The dotted lines in the two graphs are denoted for the convergence rate $O(k h^2/p^2)$ (left) and $O(k h/p)$ (right) respectively. Since the dotted lines in each graph of Figure \ref{fig-1} are parallel for different $p$, we only plot a single dotted line in each graph to reveal its dependence on $h$. The left graph of Figure \ref{fig-1} displays the relative error $\|u-u_h\|_{L^2(\Omega)}/\|u\|_{L^2(\Omega)}$ for the case $k = 200$ by the FOSLS method based on $RT1P1, RT2P2, RT3P3$ and $RT4P4$ approximations, while the right graph displays the corresponding relative error $\|\boldsymbol{\phi}-\boldsymbol{\phi}_h\|_{L^2(\Omega)}/\|\boldsymbol{\phi}\|_{L^2(\Omega)}$. We find that the relative error for $u$ converges slower than $O(k h^2/p^2)$ when $p=1$ or $2$ on the underlying meshes, however, it converges almost or faster than $O(k h^2/p^2)$ for higher polynomial degree $p=3$ or $4$. The similar phenomenon can also be observed from the right graph of Figure \ref{fig-1}  for the relative error for $\boldsymbol{\phi}$ compared with the convergence rate $O(k h/p)$.

\begin{figure}[htbp]
\centering
    \includegraphics[width=2.3in]{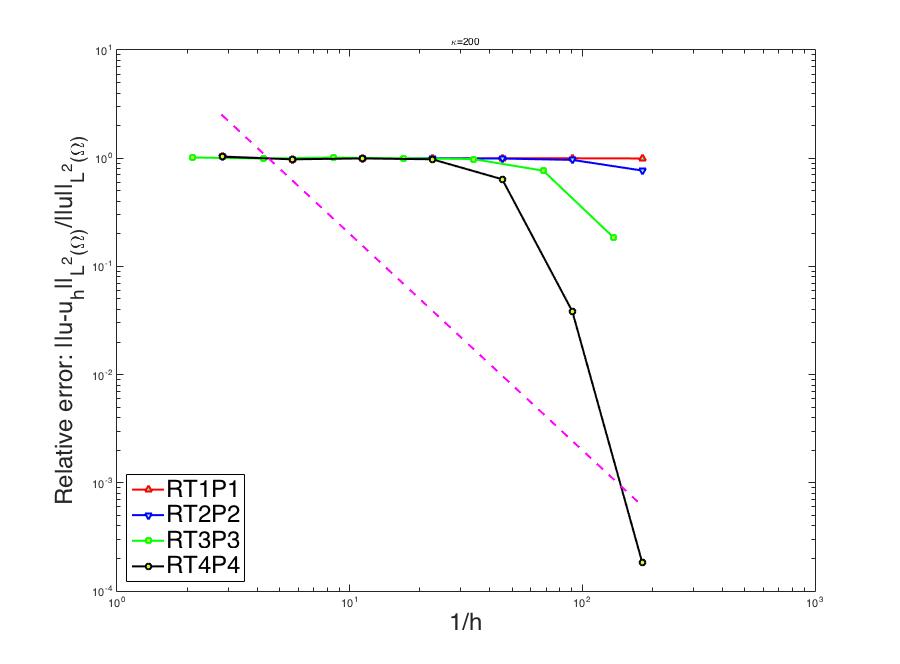}
    \includegraphics[width=2.3in]{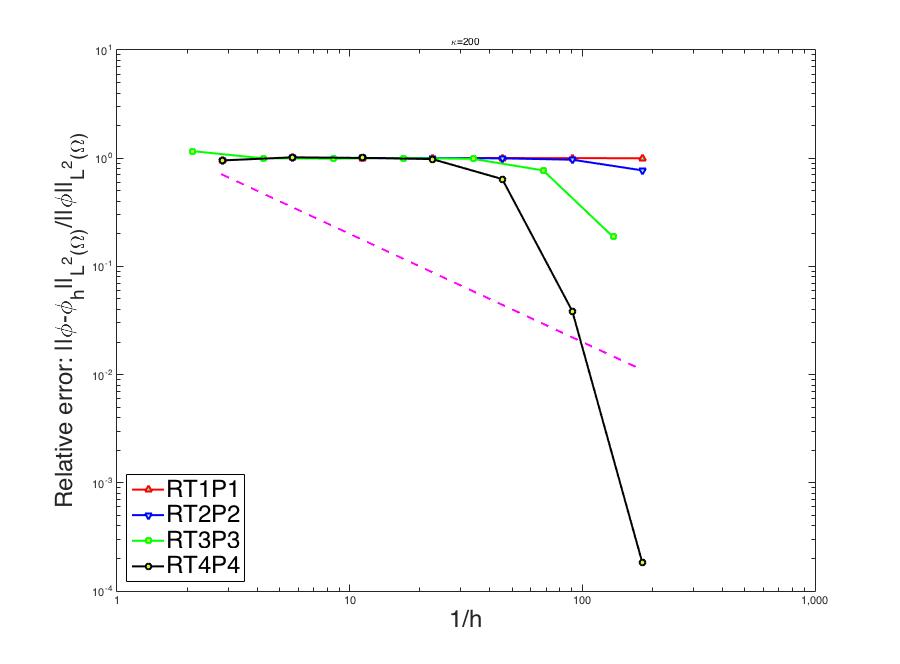}
    \caption{\footnotesize The relative errors $\|u-u_h\|_{L^2(\Omega)}/\|u\|_{L^2(\Omega)}$ (left) and  $\|\boldsymbol{\phi}-\boldsymbol{\phi}_h\|_{L^2(\Omega)}/\|\boldsymbol{\phi}\|_{L^2(\Omega)}$ (right) for the case $k=200$ based on the FOSLS method with different polynomial degree.}\label{fig-1}
\end{figure}

\begin{figure}[htbp]
\centering
    \includegraphics[width=2.3in]{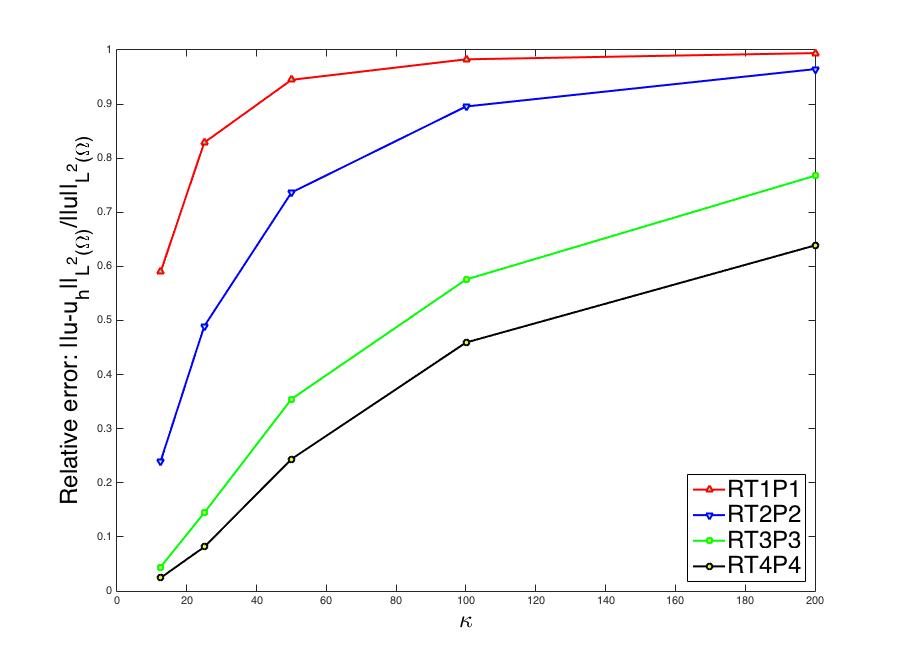}
    \includegraphics[width=2.3in]{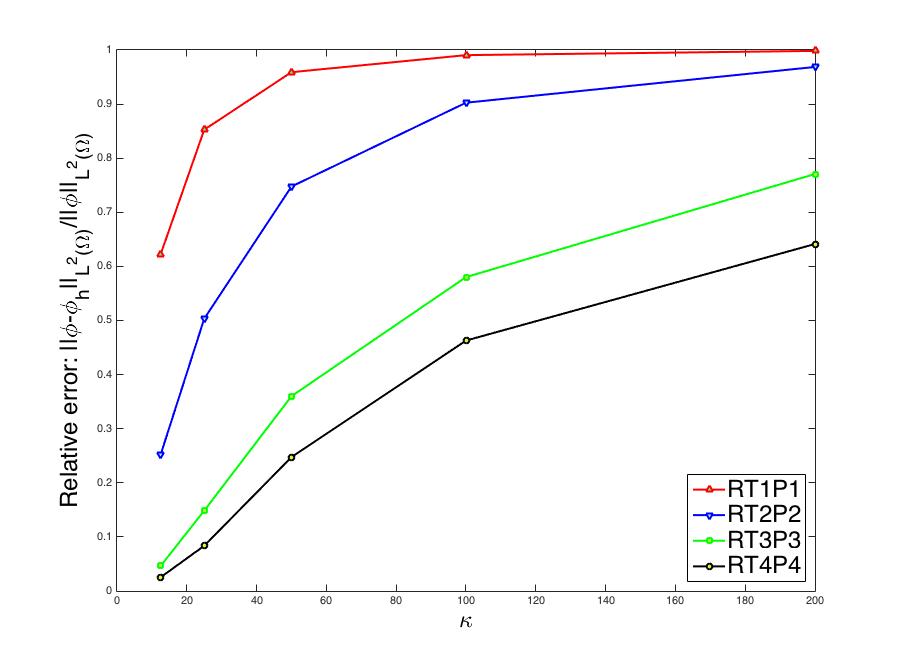}
    \caption{\footnotesize The relative errors $\|u-u_h\|_{L^2(\Omega)}/\|u\|_{L^2(\Omega)}$ (left) and  $\|\boldsymbol{\phi}-\boldsymbol{\phi}_h\|_{L^2(\Omega)}/\|\boldsymbol{\phi}\|_{L^2(\Omega)}$ (right) under the mesh condition $kh/p \approx  1$.  }\label{fig-2-1}
\end{figure}

\begin{figure}[htbp]
\centering
    \includegraphics[width=2.3in]{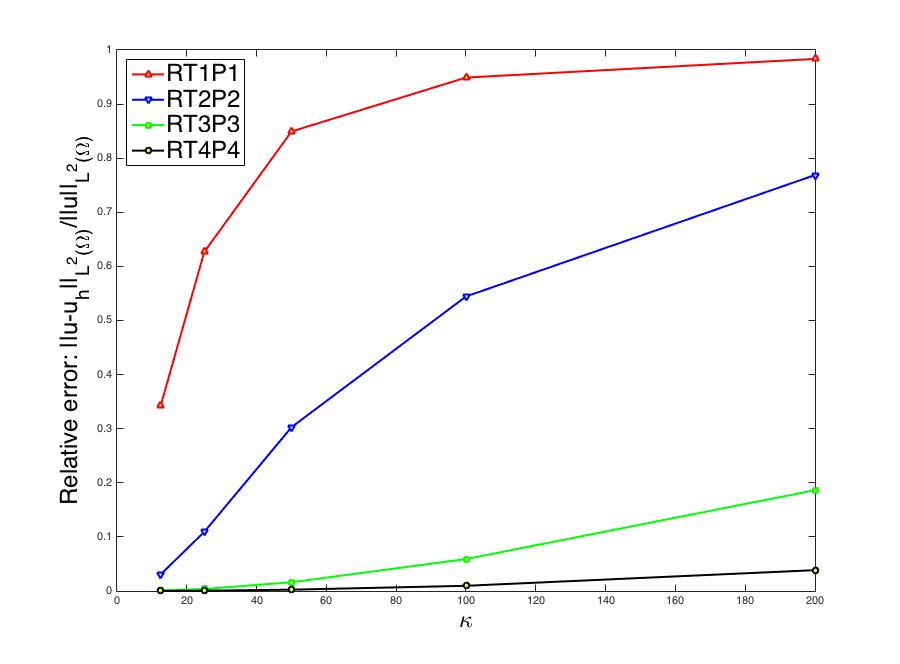}
    \includegraphics[width=2.3in]{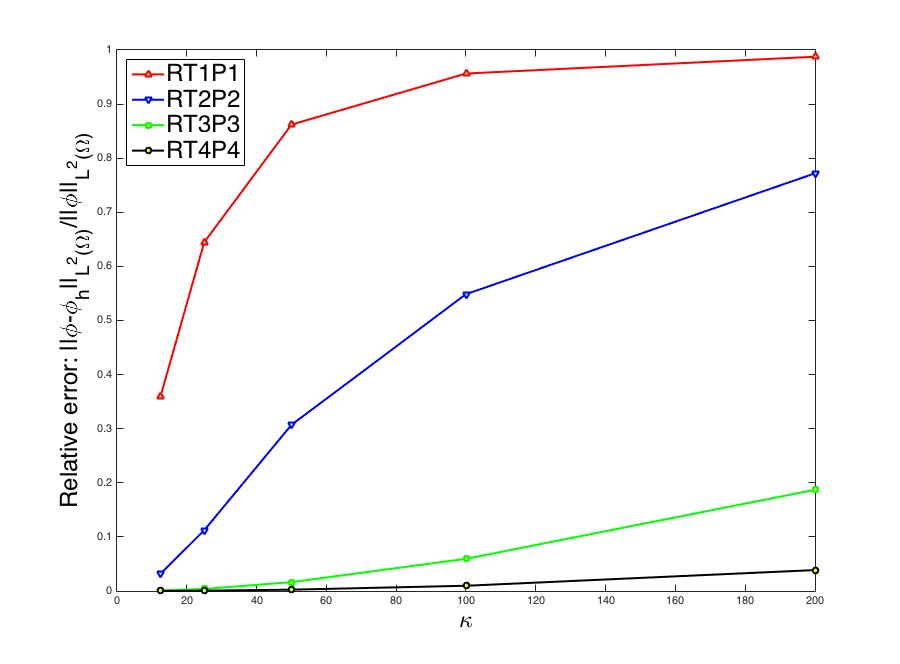}
    \caption{\footnotesize The relative errors $\|u-u_h\|_{L^2(\Omega)}/\|u\|_{L^2(\Omega)}$ (left) and  $\|\boldsymbol{\phi}-\boldsymbol{\phi}_h\|_{L^2(\Omega)}/\|\boldsymbol{\phi}\|_{L^2(\Omega)}$ (right) under the mesh condition  $kh/p \approx  0.5$. }\label{fig-2-2}
\end{figure}

Figure \ref{fig-2-1} displays the relative errors for $u$ and $\boldsymbol{\phi}$ under the mesh condition $kh/p \approx  1$ respectively. It shows that for the FOSLS method based on different polynomial degree approximations (p=1,2,3,4), both two types of relative errors cannot be controlled under the mesh condition $kh/p \approx  1$ and increase with the wave number $k$, which indicates the existence of the pollution error. Figure \ref{fig-2-2} displays the same relative errors under the mesh condition $kh/p \approx  0.5$. We observe that under this mesh condition, although the relative errors still  increase with the wave number $k$ for the FOSLS method based on lower order polynomial approximations, the relative errors are quite small for different wave number $k$ when the polynomial degree $p=4$. The results support the theoretical analysis.

\begin{figure}[htbp]
\centering
    \includegraphics[width=2.3in]{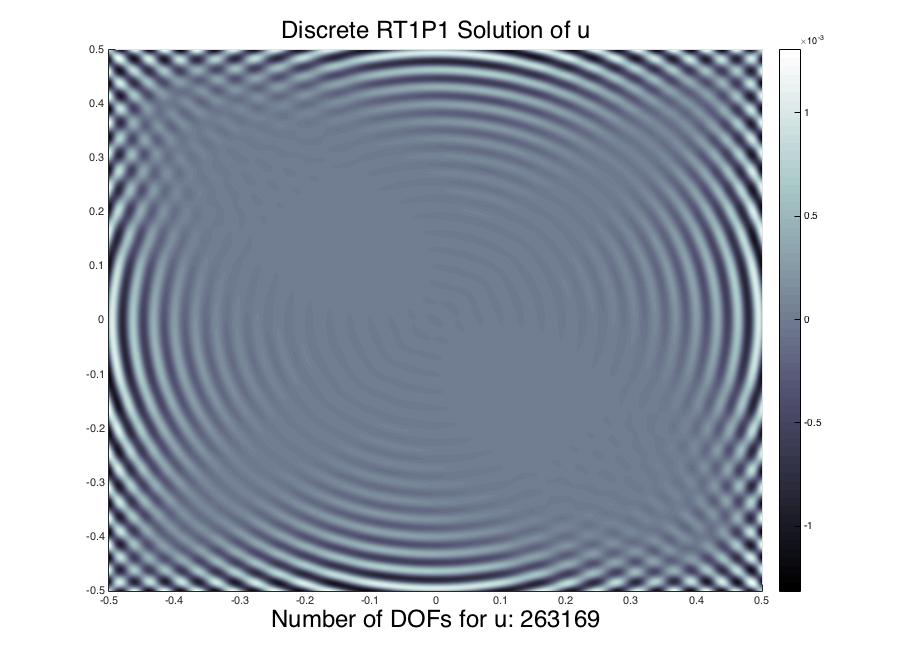}
    \includegraphics[width=2.3in]{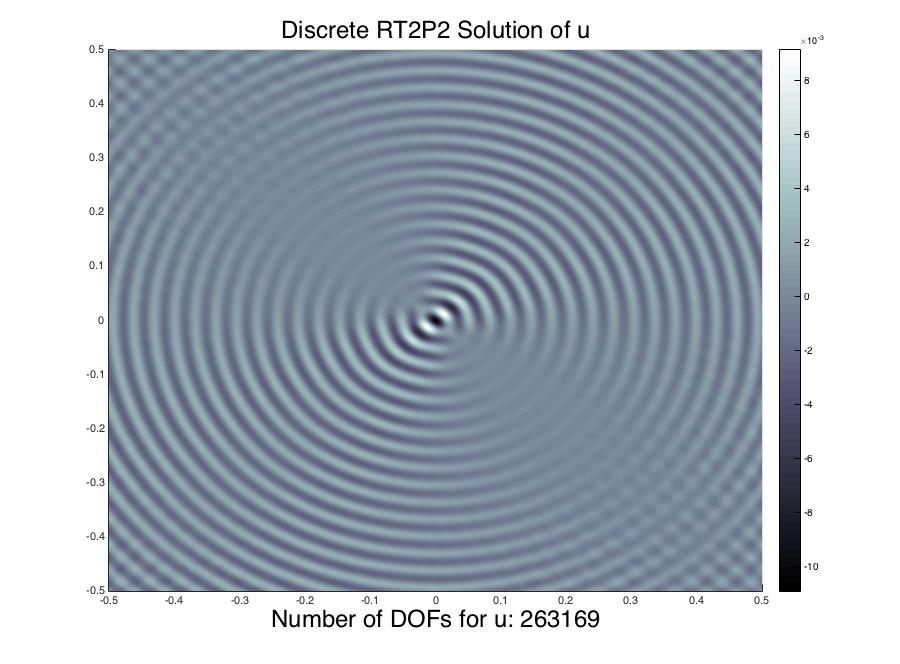}
    \includegraphics[width=2.3in]{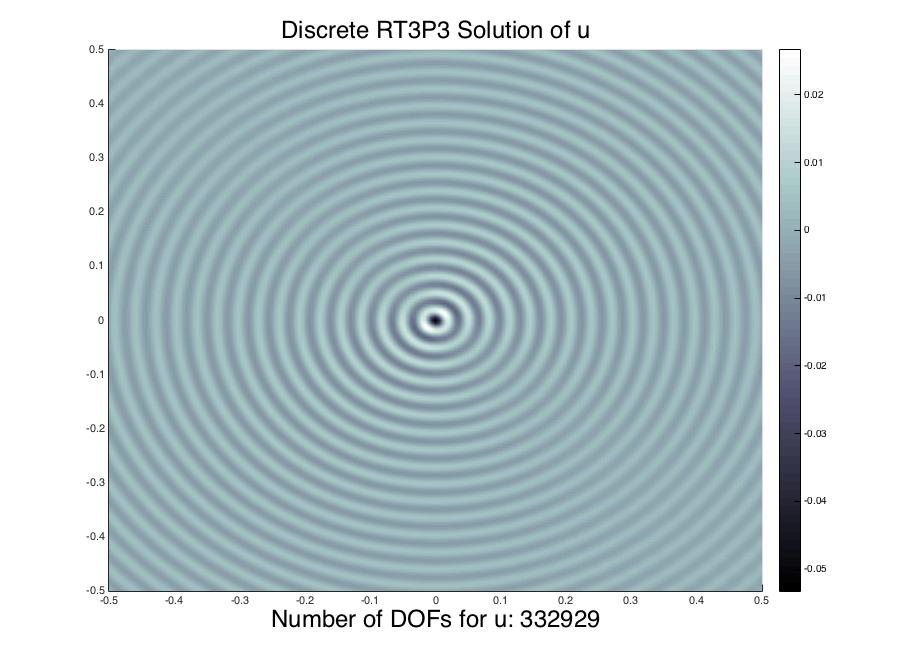}
    \includegraphics[width=2.3in]{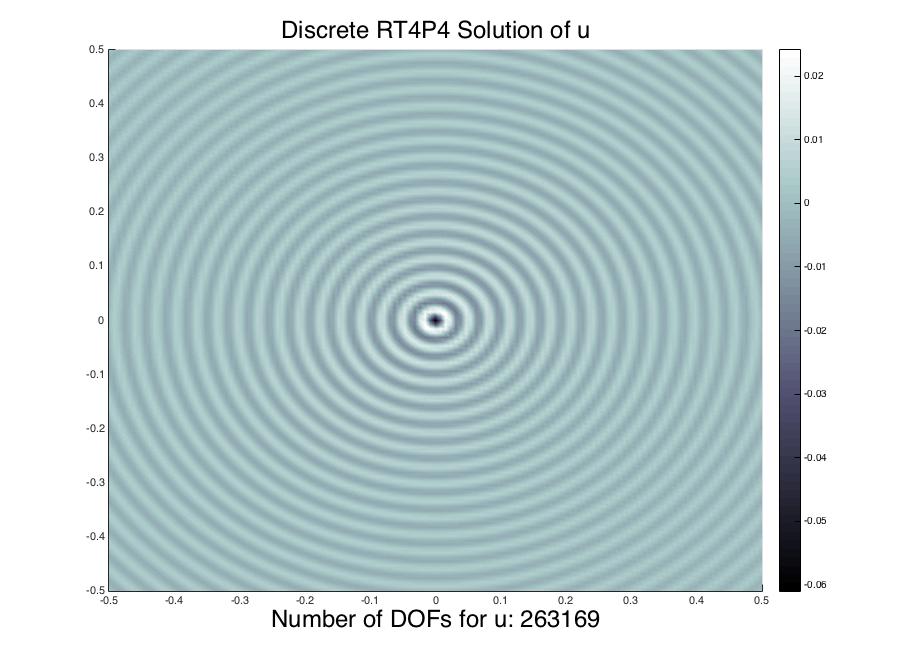}
    \caption{\footnotesize  Surface plot of the imaginary parts of the FOSLS solution $u_h$ based on $RT1P1$ (top-left), $RT2P2$ 
    (top-right), $RT3P3$ (bottom-left), $RT4P4$ (bottom-right) approximations for the case $k = 200$ under the mesh condition 
    $kh/p \approx  0.5$.}\label{fig-3}
\end{figure}

\begin{figure}[htbp]
\centering
    \includegraphics[width=2.3in]{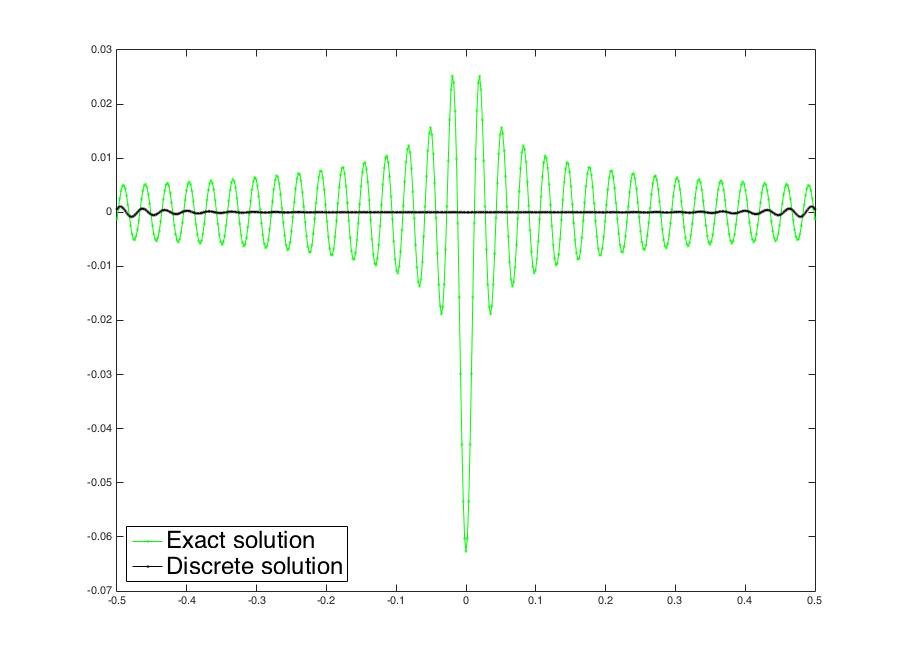}
    \includegraphics[width=2.3in]{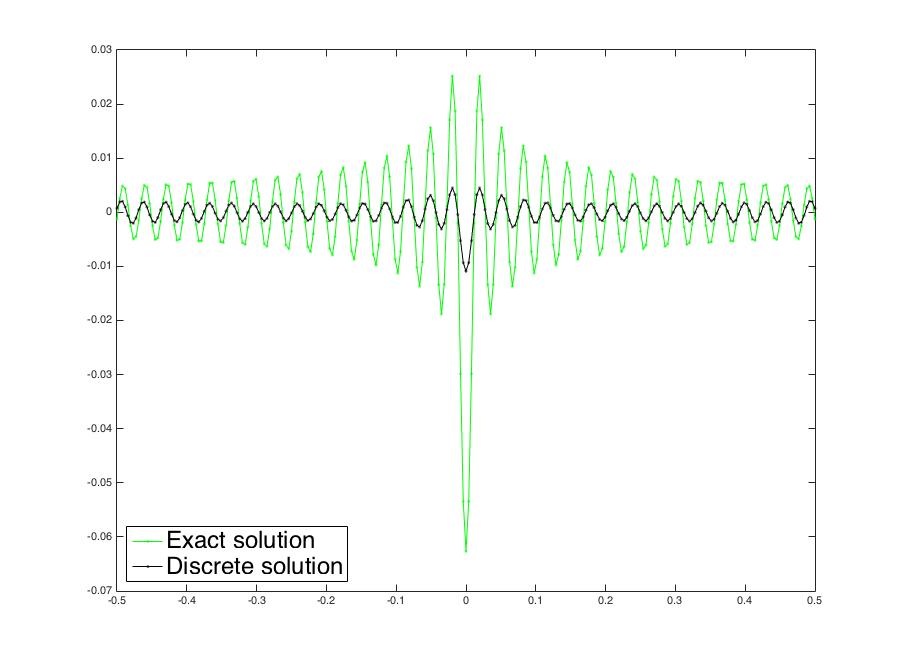}
    \includegraphics[width=2.3in]{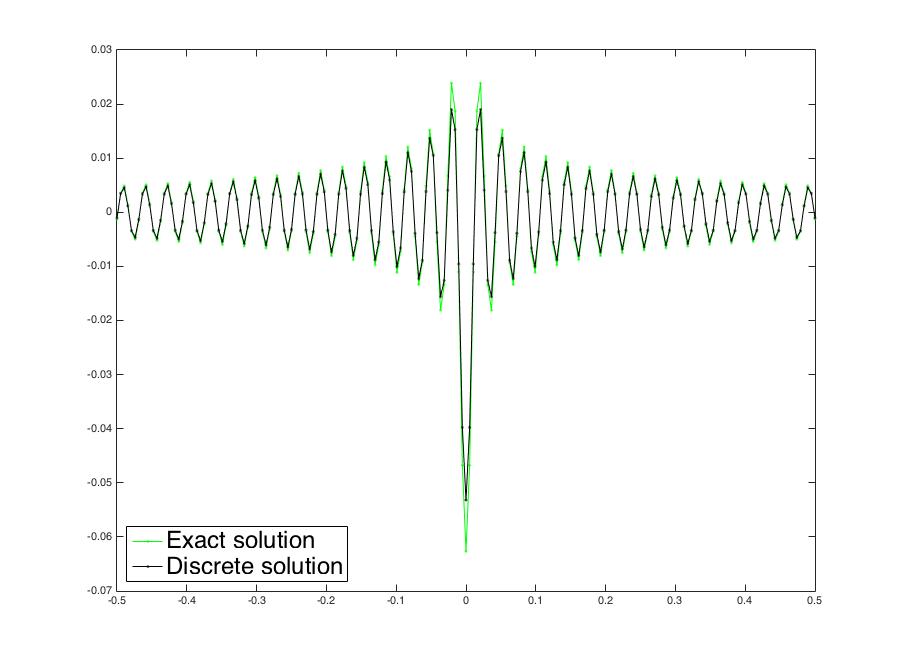}
    \includegraphics[width=2.3in]{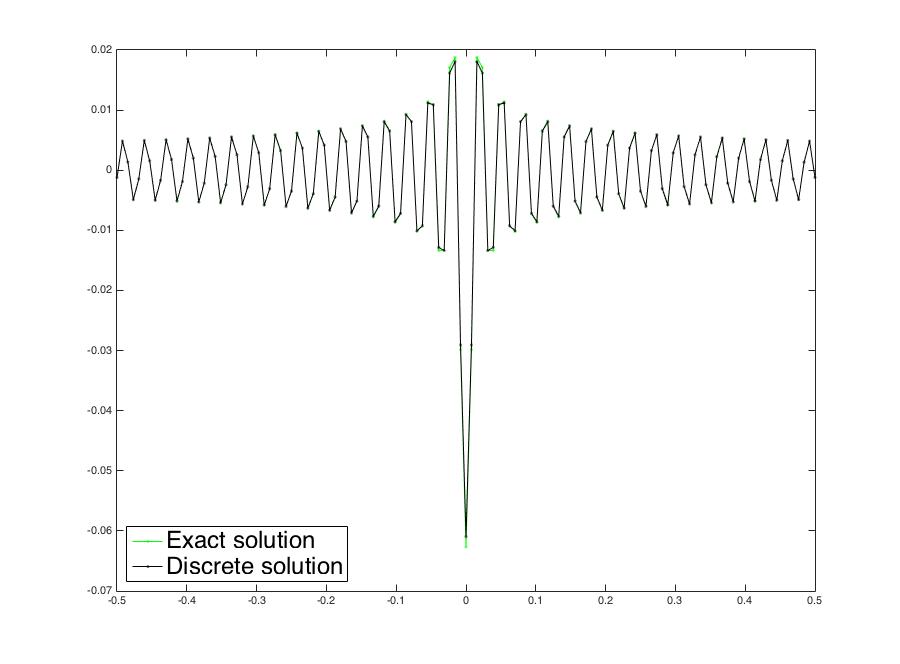}
    \caption{\footnotesize The traces of imaginary part of the FOSLS solution $u_h$ based on $RT1P1$ (top-left), $RT2P2$ 
    (top-right), $RT3P3$ (bottom-left), $RT4P4$ (bottom-right) approximations for the case $k = 200$ under the mesh 
    condition $kh/p \approx  0.5$. The trace of imaginary part of the exact solution is plotted in the green lines.}\label{fig-4}
\end{figure}

For more detailed comparison between FOSLS methods with different polynomial degree approximations, we consider 
the Helmholtz problem with wave number $k = 200$. Figure \ref{fig-3} displays the surface plots of the imaginary parts 
of the FOSLS solutions of $u_h$ based on the $RT1P1$, $RT2P2$, $RT3P3$ and $RT4P4$ approximations under mesh condition 
$kh/p \approx  0.5$. The traces of imaginary part of the FOSLS solution $u_h$ based on the $RT1P1$, $RT2P2$, $RT3P3$ 
and $RT4P4$ approximations in the $xz$-plane under mesh condition $kh/p \approx  0.5$, and the trace of imaginary part of 
the exact solution, are both shown in Figure \ref{fig-4}. It is shown that the FOSLS solutions $u_h$ based on $RT3P3$ and 
$RT4P4$ approximations have almost correct shapes and amplitudes as the exact solution, while the FOSLS solution $u_h$ 
based on low order polynomial approximations does not match the exact solution well. Thus we can observe that although 
the phase error appears in the case of low order polynomial approximation, it can be reduced by high order polynomial approximation.

\end{document}